\documentclass[journal]{IEEEtran}

\usepackage[utf8]{inputenc}
\usepackage{amssymb}
\usepackage{amsmath}
\usepackage{amsthm}
\usepackage{bm}

\newtheorem{theorem}{Theorem}

\usepackage{xcolor}
\usepackage{graphicx}
\usepackage[colorlinks,citecolor=blue]{hyperref}
\usepackage{subcaption}
\usepackage{booktabs,etoolbox} 
\usepackage{stfloats} 

\usepackage{graphicx}
\usepackage{algorithm}
\usepackage{algpseudocode}
\usepackage[acronyms, shortcuts]{glossaries}

\let\oldac\ac
\renewcommand*\ac[1]{\textcolor{black}{\oldac{#1}}}

\usepackage{multirow}

\ifCLASSINFOpdf
\else
\fi

\theoremstyle{definition}

\newtheorem{definition}{Definition}
\theoremstyle{remark}
\newtheorem{remark}{Remark}

\def\keywords{\normalfont%
    \if@twocolumn%
    \@IEEEabskeysecsize\bfseries\textit{Index Terms}---\,\relax%
}

\newcommand{\cc}[1]{\mathcal{#1}}
\newcommand{\bb}[1]{\mathbb{#1}}

\newcommand{\agenum}{N}
\newcommand{\ageset}{[\agenum]}
\newcommand{\horizon}{T}
\newcommand{\transpose}{\mathsf{T}}
\newcommand{\state}{s}
\newcommand{\action}{a}

\begin{document}


\title{Maximum-Entropy Multi-Agent Dynamic Games: Forward and Inverse Solutions}

\author{Negar Mehr$^{1}$,
        Mingyu Wang$^{2}$, and
        Mac Schwager$^{3}$
\thanks{This work was supported in part by ONR grant N00014-18-1-2830. Toyota Research Institute provided funds to support this work.}
\thanks{$^{1}$Negar Mehr is with the Aerospace Engineering department, University of Illinois Urbana-Champaign, Urbana, IL. \tt\small negar@illinois.edu}
\thanks{$^{2}$Mingyu Wang is with the Department of Mechanical Engineering, Stanford University, Stanford, CA. \tt\small mingyuw@stanford.edu} 
\thanks{$^{3}$Mac Schwager is with Department of Aeronautics and Astronautics, Stanford University, Stanford, CA. \tt\small schwager@stanford.edu}}

\maketitle

\begin{abstract}
In this paper, we study the problem of multiple stochastic agents interacting in a dynamic game scenario with continuous state and action spaces.  We define a new notion of stochastic Nash equilibrium for boundedly rational agents, which we call the Entropic Cost Equilibrium (ECE).  We show that ECE is a natural extension to multiple agents of Maximum Entropy optimality for single agents.  We solve both the ``forward'' and ``inverse'' problems for the multi-agent ECE game.  For the forward problem, we provide a Riccati algorithm to compute closed-form ECE feedback policies for the agents, which are exact in the Linear-Quadratic-Gaussian case.  We give an iterative variant to find locally ECE feedback policies for the nonlinear case.  For the inverse problem, we present an algorithm to infer the cost functions of the multiple interacting agents given noisy, boundedly rational input and state trajectory examples from agents acting in an ECE.  The effectiveness of our algorithms is demonstrated in a simulated multi-agent collision avoidance scenario, and with data from the INTERACTION traffic dataset. In both cases, we show that, by taking into account the agents' game theoretic interactions using our algorithm, a more accurate model of agents' costs can be learned, compared with standard inverse optimal control methods.

\end{abstract}


\IEEEpeerreviewmaketitle

\section{Introduction}\label{sec:intro}
In this article, we seek to learn the cost functions of a group of interacting dynamic agents from a set of trajectory demonstrations of those interactions.  We call this problem an Inverse Dynamic Game (IDG), analogously to Inverse Reinforcement Learning (IRL)  and Inverse Optimal Control (IOC) in the single-agent setting.  To solve the inverse dynamic game, we first formulate a new notion of stochastic Nash equilibrium to describe the boundedly rational equilibrium condition found in natural human demonstrations.  We call this equilibrium an Entropic Cost Equilibrium (ECE).  We then present an algorithm to find feedback policies for agents engaged in an ECE game, that is, we solve the ``forward'' problem.  We then use this forward solution within an algorithm to infer the cost functions of the multiple agents given trajectory demonstrations from those agents, i.e., we solve the ``inverse'' problem. 

For many robotic applications, it is not trivial to design a cost function that mimics an expert's behavior, such as a human driver's actions. In such applications, a miss-specified cost function may lead to undesired behaviors, and designing the correct cost function is notoriously challenging. A common practice is to infer the cost function from experts' demonstrations through the framework of Inverse Reinforcement Learning (IRL), also sometimes called Inverse Optimal Control (IOC). An IRL algorithm 
infers the cost function by observing an agent's behavior assuming that the agent behaves approximately optimally. However, real-world robotic applications, such as autonomous driving, usually involve multiple interactive agents whose behaviors are coupled through the feedback interactions between the agents. Consequently, when learning from interactive agents' behaviors, we cannot treat them as acting in isolation. Instead, we need to take into account the game theoretic coupling between agent's behaviors. In this paper, we develop inverse methods for such interactive multi-agent settings, where we learn each agent's cost function while taking into account their feedback interactions. We call this an inverse dynamic game (IDG).

One of the main challenges in solving IDGs is that agents no longer optimize their own cost functions in isolation. Rather, they reach a notion of game theoretic equilibrium. Hence, agents' cost functions must be learned such that the learned cost functions rationalize the set of demonstrations as game theoretic equilibrium strategies, rather than optimal strategies. Moreover, when learning from experts such as humans, we need to account for the humans' noisy behavior and bounded rationality. Reaching exact equilibria requires perfect rationality of agents; however, in decision-making settings, humans' rationality is normally bounded due to the limited information they have, their cognitive limitations and the finite amount of time available for decision making. Thus, we need to account for the noise in humans' decision making in a multi-agent setting. 

In this paper, we address these challenges by defining a notion of ``noisy" equilibrium for capturing the outcome of interactions between multiple boundedly rational agents. We call this equilibrium an \emph{Entropic Cost Equilibrium} (ECE) and show its connections to the maximum entropy framework common in IRL~\cite{ziebart2008maximum}. We prove that the ECE concept is indeed an extension of maximum-entropy optimality to multi-agent settings. 
Once we formalize the notion of Entropic Cost Equilibrium, to assess the quality of a set of learned agents' cost functions, one must find the ECE policies under the learned costs and compare them to the set of agents' demonstrations. 
To enable such comparisons, we develop an algorithm to find ECE policies for a given set of agents' cost functions. We prove how ECE policies can be obtained in closed-form for a special class of games, namely linear-quadratic-Gaussian games, using a Riccati solution reminiscent of the well-known Linear-Quadratic Regulator (LQR) and Linear-Quadratic Game (LQGame) solutions. 
Leveraging this result, we provide an iterative algorithm for approximating ECE policies in general multi-agent games with general nonlinear dynamics and costs. 

Knowing how to approximate ECE policies for a set of given costs, we propose a multi-agent inverse dynamic game algorithm for learning agents' costs. Similar to the common practice in IRL~\cite{ziebart2008maximum,ziebart2010modeling,levine2012continuous,hadfield2016cooperative}, we assume that each agent's cost function is parameterized as a linear combination of a set of known features. We then propose an iterative algorithm for learning the weights of the features in each agent's cost function such that the feature expectations under the learned costs match the empirical feature average from the demonstrations. 

To demonstrate the effectiveness of our algorithm, we verify our algorithm using both synthetic datasets and real-world traffic data. First, we consider a goal-reaching and collision-avoidance scenario involving two and three agents. We show that by taking into account the agents' feedback interactions, a more accurate model of agents' costs can be learned. We then validate the performance of our algorithm on the INTERACION dataset~\cite{interactiondataset} which involves highly interactive multi-agent driving scenarios collected from different countries. We again demonstrate that by taking into account the agents' interactions through our inverse dynamic game framework, more accurate predictions of agents' behavior can be made. We show that the prediction accuracy of our framework is very close to the intelligent driver's model (IDM) \cite{treiber2000congested}, which is a highly accurate human-designed driving model used for modeling human drivers' leader-following behavior.

We summarize our contributions as follows:
\begin{itemize}
\item We define and formalize the notion of Entropic Cost Equilibrium (ECE) for capturing the interaction of noisy agents. 
\item We develop an algorithm for computing approximate ECE policies for general nonlinear multi-agent systems (the forward problem).
\item We propose an iterative algorithm for learning the agents' costs from a set of interactive demonstrations (the inverse problem).
\item We validate our proposed inverse dynamic game algorithm in both synthetic and real-world driving scenarios. 
\end{itemize}

The organization of this paper is as follows. Section~\ref{sec:rel} provides an overview of the related work. In Section~\ref{sec:prelim}, we introduce our notation and discuss the preliminaries. In~\ref{sec:equ-char}, we define the notion of Entropic Cost Equilibrium and discuss its connection to the maximum-entropy framework. We discuss finding ECE policies for general nonlinear multi-agent systems in Section~\ref{sec:eq-policies}.  Section~\ref{sec:ma-irl} provides the description of our inverse dynamic game algorithm.
Simulations and experiments with the analysis of the resulting performance are incorporated in Section~\ref{sec:experiments}. Finally, we conclude the paper and discuss future directions in Section~\ref{sec:concl}.

\section{Related Work}\label{sec:rel}
\subsection{Single-Agent Inverse Reinforcement Learning}
Inferring and learning cost functions from system trajectories has been widely studied for single-agent systems. The problem of inferring an agent's cost function was first studied by Kalman in the context of inverse optimal control for linear quadratic systems with a linear control law \cite{kalman1964linear}. Learning an agent's cost was later studied in~\cite{ng2000algorithms} and~\cite{abbeel2004apprenticeship} where the assumption was that the demonstrations satisfy optimality conditions. This assumption was relaxed to take into account the bounded rationality and human's noisy demonstrations in the framework of maximum entropy inverse reinforcement learning in \cite{ziebart2008maximum}.

Despite the success of these IRL methods, they were mostly developed for discrete state and action spaces. In~\cite{levine2012continuous}, maximum-entropy cost inference was studied for systems with continuous state and action spaces. The common assumption in these works is that the agent's cost function is paramterized as a weighted sum of features. This assumption was further relaxed in~\cite{finn2016guided}, where the underlying cost function was learned as a neural network in a maximum-entropy framework. In~\cite{ho2016generative}, a framework was proposed for directly extracting a policy from an agent's demonstrations, as if it were obtained by reinforcement learning following maximum entropy inverse reinforcement learning. 
Inferring an agent's cost function was also studied as a bilevel optimization where, in the outer loop, the cost parameters were found such that the estimation error of system trajectories or the likelihood of demonstrations was optimized~\cite{albrecht2011imitating,mombaur2010human}. Recent works have proposed learning cost parameters that minimize the residual of Karush–Kuhn–Tucker (KKT) conditions at the demonstrations~\cite{englert2017inverse,awasthi2019forward,menner2020maximum}. Our work draws inspiration from the Maximum Entropy IRL framework where the noisy behavior of the agent is captured. 
\subsection{Multi-Agent Inverse Reinforcement Learning}
Extension of single-agent IRL algorithms to the multi-agent settings have been studied in the past. Multi-agent IRL was studied for discrete state and action spaces in~\cite{natarajan2010multi}. The inverse equilibrium problem was considered in~\cite{waugh2013computational} where the maximum entropy principle and the notion of regret were employed for solving the inverse equilibrium problem. No systems dynamics were considered in~\cite{waugh2013computational}, and the inverse equilibrium problem was considered in matrix games. Inverse reinforcement learning was considered  in~\cite{lin2014multi} for two-person zero-sum stochastic games with discrete state and action spaces where the cost learning problem was formulated as an optimization problem. This framework was later extended to general-sum games with discrete state and action spaces in~\cite{lin2019multi}. In~\cite{yu2019multi}, a new framework for multi-agent inverse reinforcement learning utilizing adversarial machine learning was proposed for high-dimensional state and action spaces with unknown dynamics. In~\cite{song2018multi},
generative adversarial imitation learning in the single agent case~\cite{ho2016generative} was extended to the multi-agent setting. The current work is distinct in that it focuses on  multi-agent IRL in general-sum games with \emph{known} system dynamics and continuous state and action spaces.


Multi-agent IRL has been studied as an estimation problem too. In~\cite{schwarting2019social}, a particle filtering algorithm  was utilized for online estimation of human behavior parameters where the critical role of accurate human motion prediction was demonstrated. In~\cite{lecleac2021lucidgames}, a filtering technique based on an unscented Kalman filter was developed for online estimation of cost parameters in multi-agent settings. In~\cite{kopf2017inverse}, inverse reinforcement learning was considered for the class of linear quadratic games where the equilibrium strategies of all but one agent were known. This assumption simplified the problem and reduced it to effectively an instance of the single-agent cost inference problem.~\cite{rothfuss2017inverse} extended this by proposing to minimize the residuals of the first-order necessary conditions for open-loop Nash equilibria. 
In~\cite{inga2019inverse}, residual errors of optimality conditions for open-loop Nash equilibria were minimized in a maximum-entropy framework. In~\cite{awasthi2020inverse}, state and input constraints are also considered in a maximum-entropy residual-minimization framework. In~\cite{peters2021inferring}, agents' cost functions were estimated under partial observability. In this work, we focus on cost learning for general \emph{nonlinear} games, and find cost parameters which rationalize interactions under \emph{feedback} information structure. We further capture the \emph{noisy} behavior of humans with our Entropic Cost Equilibrium concept. In~\cite{inga2019inverse}, maximum-entropy MA-IRL was considered as a maximum likelihood problem. However, when computing the probability distribution over agents' trajectories, the coupling between the agents was not considered. Moreover, in this work, it was assumed that agents' feedback policies were known a-prioi which is a restrictive assumption. In this paper, we show how the maximum-entropy framework can be formalized in a multi-agent game theoretic setting to account for agents' feedback interactions. We further provide an algorithm which does not require knowledge of the agents' policies and verify it on a real-world data set.

\section{Preliminaries}\label{sec:prelim}

Consider $\agenum \geq 1$ agents interacting in an environment. We use $\ageset=\{1,\cdots,\agenum\}$ to refer to the set of all agents' indices. Let $s_t\in \cc{S} $ denote the vector of joint states of all agents at each time $t$, where the set $\cc{S} \subseteq \bb{R}^n$ is the joint state space observed by all agents, and $n$ is the dimension of the state space. Each agent $i\in \ageset$ decides on its action $a^i_t \in \cc{A}^i$ at time step $t$, where $\cc{A}^i \subseteq \bb{R}^{m^i} $ is the action space of agent $i$, and ${m^i}$ is the dimension of the action space of agent $i$. Throughout this paper, we use bold letters to refer to the concatenation of variables for all agents. For a given time step $t$, we use $\bm{a}_t = (a_t^1,\cdots,a_t^\agenum)$ to denote the

vector of all agents' actions at time $t$. Following the conventional notation used in the game theory literature, we utilize the superscript $-i$ to indicate all agents expect agent $i$. For example, ${\bm{a}}_t^{-i}$ represents the vector of all agents' actions excluding the action of agent $i$ at time $t$. We define $\bm{\cc{A}}=\{\cc{A}^i\}_{i=1}^\agenum$ to denote the collection of the action spaces of all the agents. 

We assume that agents choose their actions through a stochastic Markovian policy. For each agent $i$, we use $\pi^i_t(.|\state_t)$ to denote such a policy for agent $i$ where $\pi^i_t(\action^i_t|\state_t)$ encodes the probability of agent $i$ taking the action $\action^i_t$ at time $t$ given that the system is in state $\state_t$. Note that we are assuming that each agent selects its action independently i.e., $\pi^i_t(\action^i_t | \state_t, \bm{\action}^{-i}_t) = \pi^i_t(\action^i_t | \state_t)$. For a finite horizon $\horizon$, we use $\pi^i=\{\pi^i_t\}_{t=1}^\horizon$ to refer to the agent $i$'s policy for the entire horizon $\horizon$. Moreover, we use $\bm{\pi} = \{ \pi^i\}_{i=1}^\agenum$ to refer to the set of all agents' time-dependent policies. The discrete-time dynamics of the system are represented by state updates of the following form
\begin{align}
\begin{split}
    \state_{t+1} &= f(t, \state_t, \bm{\action}_t) + g(\state_t)w_t, \\
    \state_1 &\sim p_1(\state_1), \quad w_t \sim p_w.
    \end{split}
    \label{eq:dynamics}
\end{align}

where $p_1$ and $p_w$ are the distributions of the system initial state and system noise respectively. We assume that each agent $i$ has a bounded per-stage cost function $c^i: \cc{S} \times \cc{A}^1  \cdots \times \cc{A}^\agenum \rightarrow \bb{R}$. We further let $\bm{c}=\{c^i\}_{i=1}^\agenum$ represent the vector of all agents' per-stage costs.
We assume that each agent $i$ is seeking to minimize its own expected cumulative cost $\bb{E}_{{\bm{\pi}}} \sum_{t=1}^\horizon c^i(\state_t,\bm{\action}_t)$.  

We model the agents' interaction as a dynamic game. We use the notation $G=(\cc{S},\bm{\cc{A}}, f, g, \bm{c},\horizon)$ to refer to a game with a time horizon of length $\horizon$ between $\agenum$ agents whose action spaces and costs are defined via $\bm{\cc{A}}$ and $\bm{c}$ on the state space $\cc{S}$ with the prespecified dynamics~\eqref{eq:dynamics}.

It is well known that the outcome of interaction between \emph{perfectly} rational agents is best represented via Nash equilibria of the underlying game. Given a game $G=(\cc{S},\bm{\cc{A}}, f, g, \bm{c},\horizon)$, a set of (mixed strategy) Nash equilibrium policies are defined through the following. 

\begin{definition}\label{def:nash}
Given a game $G=(\cc{S},\bm{\cc{A}}, f, g, \bm{c},\horizon)$, a set of agents' policies $\bm{\pi}^*$ is a (mixed-strategy) Nash equilibrium if and only if for each agent $i \in \ageset$, we have
\begin{align}\label{eq:nash-def} 
\begin{split}
    \bb{E}_{\pi^{i^*},\bm{\pi}^{{-i}*}} &\sum_{k=1}^\horizon c^i(\state_k,\action_k^i,\bm{\action}_k^{-i}) \leq \\ &\bb{E}_{\pi^{i},\bm{\pi}^{{-i}^*}} \sum_{k=1}^\horizon c^i(\state_k,\action_k^i,\bm{\action}_k^{-i}), \,\, \forall \pi^{i}.
\end{split}
\end{align}
\end{definition}
\noindent Definition~\ref{def:nash} implies that, at a Nash equilibrium, no agent will reduce its accumulated cost by unilaterally changing its policy from the equilibrium policy $\pi^{i^*}$ to another policy $\pi^i$. 

Although Nash equilibrium is a powerful concept for modeling the interaction of agents, achieving Nash euilibria requires perfect rationality of agents. Moreover, it is well known that computing Nash equilibria is in general intractable even in normal-form games~\cite{daskalakis2009complexity}. Thus, when learning from a set of demonstrations that are collected from experts such as humans, assuming that humans have achieved a Nash equilibrium may be unreasonable.  Not only are humans computationally bounded, but they also may act under noisy information, or produce actions that are different than what they intend, making them appear to act irrationally to some degree. This concept is known in game theory as \emph{bounded rationality}.  Instead of making the ``best" choices, humans often make choices that are "satisfactory on average"~\cite{simon1955behavioral,simon1979rational}. In the next section, we generalize the notion of Nash equilibrium to capture the bounded rationality of experts such as humans during their interactions.

\section{Entropic Cost Equilibrium}\label{sec:equ-char}
In the seminal work of~\cite{mckelvey1995quantal,mckelvey1998quantal}, it was shown that to take into account the bounded-rationality of humans, their interaction can be modeled via the notion of quantal response equilibrium in normal form and extended games. It was demonstrated that the quantal response equilibrium can successfully model humans' choices in a set of lab experiments, while the predictions made by the Nash equilibrium deviated largely from the lab experiments. At quantal response equilibrium, every agent maintains a probability distribution over its actions. In fact, in quantal response equilibrium, the noisy behavior of humans is captured, where the probability of an action taken by a human is related to the cost associated to that action for the human. In \cite{anderson2004noisy}, a continuous version of the notion of quantal response equilibrium, called logit equilibrium, was developed for repeated continuous games from the perspective of evolutionary game theory. In this setup, at logit equilibrium, every agent computes its expected cost with respect to the probability distribution over actions of all agents. Each agent takes actions that are exponentially proportional to the negative of this expectation. In the following, we extend this notion of logit equilibrium to dynamic games with continuous state and action spaces. We use the multi-agent extension of Q functions and prove certain properties of this notion of equilibrium. 

Given a game $G=(\cc{S},\bm{\cc{A}}, f, g, \bm{c},\horizon)$, at every time step $t < T$, for each agent $i \in \ageset$, we define the quality of a state $\state_t$ and a vector of agents' actions $\bm{\action}_t$ under a given set of agents' policies $\bm{\pi}$ via the following
\begin{align}\label{eq:Q-fun}
\begin{split}
Q^i_{t, \bm{\pi}}(\state_t,\bm{\action}_t) = \,\,
  & c^{i}(\state_t, \bm{\action_t}) + \\
  & \bb{E}_{\state_{t+1:T},\bm{\pi}} \sum_{k=t+1}^T \left[ c^i(\state_k,\bm{\action}_k) | \state_t, \bm{\action}_t \right].
\end{split}
\end{align}

For the final time step $\horizon$, the cost associated with a state $\state_\horizon$ and a vector of actions $\bm{\action}_\horizon$ for an agent $i$ is
\begin{align}\label{eq:Q-fun-T}
Q^i_{T,\bm{\pi}}(\state_\horizon,\bm{\action}_\horizon) =  c^i(\state_\horizon,\bm{\action}_T) .
\end{align}
Equations~\eqref{eq:Q-fun} and~\eqref{eq:Q-fun-T} are the extensions of the definition of the Q function to the multi-agent game theoretic setting. Following~\cite{anderson2004noisy}, for each agent $i \in \ageset$ and every time step $t\leq \horizon$, we define
\begin{align}\label{eq:Qbar-fun}
\bar{Q}^i_{t,\bm{\pi}}(\state_t,\action^i_t) &= \bb{E}_{\bm{\action}^{-i}\sim \bm{\pi}^{-i}} Q^i_{t,\bm{\pi}}(\state_t,\bm{\action}_t).
\end{align}
In~\eqref{eq:Qbar-fun}, the expectation of the Q function is computed with respect to the actions of all the other agents $\bm{a}^{-i}$.

Note that $\bar{Q}^i_{\bm{\pi}}(s_t,a^i_t)$ depends only on the action of agent $i$ and the state, not the actions of the other agents. In fact, $\bar{Q}^i_{\bm{\pi}}(s_t,a^i_t)$ determines the quality of a pair of the system state and an agent's action given the set of other agents' policies $\bm{\pi^{-i}}$. Ideally, if agents were perfectly rational, given knowledge of the other agents' policies, at every time step, each agent would have taken actions that minimize $\bar{Q}^i_{\bm{\pi}}(s_t,a^i_t)$ at equilibrium. However, when agents are boundedly rational, we propose that agents' noisy behavior can be modeled through the following notion of equilibrium.



\begin{definition} For a given game 
$G=(\cc{S},\bm{\cc{A}}, f, g, \bm{c},\horizon)$, a set of agents' policies $\bm{\pi}^*=\{{\pi^i}^*\}_{i=1}^\agenum$ is an entropic cost equilibrium (ECE) if and only if for every agent $i \in \ageset$ and every action $a_t^i$, the following holds at every time step $t\leq \horizon$: 
\begin{align}\label{eq:equ-def}
{\pi^i_t}^*(\action_t^i|\state_t) = \frac{e^{-\bar{Q}^i_{t,\bm{\pi}^*}(\state_t,\action^i_t)}}{\int e^{-\bar{Q}^i_{t,\bm{\pi}^*}(\state_t,{\tilde{\action}_t}^i)} d{\tilde{\action}_t}^i }.
\end{align}
\end{definition}
\noindent Note that~\eqref{eq:equ-def} must hold for all agents. An equilibrium policy $\bm{\pi}^*$ is indeed the fixed point of~\eqref{eq:equ-def}, and $\bar{Q}^i_{\bm{\pi}^*}(s_t,a^i_t)$ depends on the policy of all agents. It is important to note that in entropic cost equilibrium (ECE), at every time step $t$, the agents' actions $a_t^i$ are independent.  No agent needs to know the action choice of any other agent to choose its own action. However, the probability of taking an action is implicitly dependent on the \emph{policies} of other agents ${\bm{\pi}}^{{-i}^*}$ through the expectation with respect to agents' policies in $\bar{Q}^i_{\bm{\pi}^*}(s_t,a^i_t)$. Therefore, although the agents' instantanous actions are independent, their policies, i.e. the probability distribution over their actions, are related in ECE through~\eqref{eq:equ-def}, and the ECE policies $\bm{\pi}^*$ are indeed the fixed points of~\eqref{eq:equ-def} for all agents.  This is in contrast to the definition of equilibrium in~\cite{yu2019multi} where a stochastic version of correlated equilibrium was developed, and the actions were assumed to be correlated. 

We would like to highlight the connection between~\eqref{eq:equ-def} and the maximum entropy framework. If there exists only one single agent in the environment,~\eqref{eq:equ-def} reverts to the well-known maximum entropy formulation, which is widely used in the development of both reinforcement learning and inverse reinforcement learning algorithms~\cite{levine2012continuous,levine2013guided, levine2014learning, haarnoja2017reinforcement,schulman2017equivalence}. If there is only one agent, the probability of taking an action $\action_t$ given a state $\state_t$ is proportional to the exponential of the negative accumulated cost from that state $(\state_t,\action_t)$. 

Now that we have defined ECE as our equilibrium notion for capturing humans' noisy interactions, we will prove a property of ECE which demonstrates its applicability and relevance to modeling the interaction of multiple agents with bounded rationality.

\begin{theorem}\label{theorem:max-ent-equi}
For a game $G=(\cc{S},\bm{\cc{A}}, f, g, \bm{c},\horizon)$, a set of agents' policies ${\bm{\pi}}^* = \{{\pi^i}^*\}_{i=1}^\agenum$ is an ECE if and only if $\bm{\pi}^*$ is a (mixed-strategy) feedback Nash equilibrium for the maximum entropy game $\tilde{G}=(\cc{S},\bm{\cc{A}}, f, g, \tilde{\bm{c}},\horizon)$ where $\tilde{\bm{c}}=\{\tilde{c}^i\}_{i=1}^\agenum$ is defined 
as
\begin{align}\label{eq:equ-max-ent}
    \tilde{c}^i(\state_t,\bm{\action}_t) = c^i(\state_t,\bm{\action}_t) - \cc{H}(\pi^i(\cdot|\state_t)),
\end{align}
and $\cc{H}(\pi^i_t(\cdot|\state_t))$ is the entropy of policy $\pi^i(\cdot|\state_t)$.  

\end{theorem}

\begin{proof}
First, we show that if $\bm{\pi}^*$ is a feedback Nash equilibrium for the game $\tilde{G}$, then, $\bm{\pi}^*$ is an ECE for the game $G$. Let $\bm{\pi}^*$ be a feedback Nash equilibrium of game $\tilde{G}$. Fix an agent $i \in \ageset$ and the policy of all the other agents $\bm{\pi}^{{-i}^*}$. Then, at Nash equilibrium of game $\tilde{G}$, the agent's policy ${\pi}^{i^*}$ optimizes the following
\begin{align}\label{eq:nash-1}
    \min_{\pi^i} \, \mathbb{E}\sum_{k=1}^T \left(c^i(\state_k,\action_k^{i},{\bm{\action}_k^{{-i}^*}}) - \mathcal{H}(\pi^i_k(\cdot|\state_k)) \right),
\end{align}
where the expectation is with respect to ${\bm{\action}_k^{{-i}^*}\sim \bm{\pi}_k^{{-i}^*}}$, and ${\action_k^{i} \sim \pi_k^i}$.
We can solve the above for finding Nash equilibrim policies $\pi^{i^*}$ using dynamic programming. Starting from the final time step $\horizon$, for each state $\state_\horizon$, we have
\begin{align}
\min_{\pi^i_\horizon} \mathbb{E} \left( c^i(\state_\horizon,\action_\horizon^{i},\bm{\action}_\horizon^{{-i}^*}) - \mathcal{H}(\pi_\horizon^i(\cdot|\state_\horizon)) \right), \label{eq:eq:rew}
\end{align}
where the expectation is with respect to ${\bm{\action}_\horizon^{{-i}^*}\sim \bm{\pi}_\horizon^{{-i}^*}}$, and ${\action_\horizon^{i} \sim \pi_\horizon^i}$. \eqref{eq:eq:rew} can be rewritten as
\begin{align} \label{eq:eq:exp-sep}
\min_{\pi^i_\horizon} \mathbb{E}_{ \pi^i_\horizon} \left(  \mathbb{E}_{{\bm{\pi}_\horizon^{{-i}^*}}} \big[ c^i(\state_\horizon,\action_\horizon^{i},\bm{\action}_\horizon^{{-i}^*}) \big] - \mathcal{H}(\pi_\horizon^i(\cdot|\state_\horizon))\right).
\end{align}
Define the following
\begin{align}
 \bar{c}^i(\state_\horizon,\action_\horizon^i)= \mathbb{E}_{\bm{\action^{{-i}^*}}_\horizon\sim{\bm{\pi}_\horizon^{{-i}^*}}} \left( c^i(\state_\horizon,\action_\horizon^{i},\bm{\action}_\horizon^{{-i}^*}) \right).   
\end{align}
Then,~\eqref{eq:eq:exp-sep} can be rewritten as minimizing the following:
\begin{align} \label{eq:eq:exp-sep-1}
&\mathbb{E}_{ \pi^i_\horizon} \left[ \bar{c}^i(\state_\horizon,\action_\horizon^i)  - \mathcal{H}(\pi_\horizon^i(\cdot|\state_\horizon)) \right] =\\
&  \mathbb{E}_{ \pi^i_\horizon} \left[ \bar{c}^i(\state_\horizon,\action_\horizon^i)  + \log(\pi_\horizon^i(\cdot|\state_\horizon)) \right] \label{eq:cost-to-KL}
\end{align}
Now, following the results in~\cite{levine2018reinforcement},~\eqref{eq:cost-to-KL} can be rewritten as
\begin{align}\label{eq:cost-KL}
\begin{split}
      D_{\text{KL}} \Big(\pi(.|s_t) \big\Vert \frac{1}{\exp{\left(V^i_\horizon(\state_\horizon)\right)}}& \exp{\left(-\bar{c}^i(\state_\horizon,\action_\horizon^i)\right)} \Big) + V^i_\horizon(\state_\horizon),
    \end{split}
\end{align}
where $D_{\text{KL}}$ denotes the KL divergence, and $V^i_\horizon(\state_\horizon)$ is
\begin{align}
    V^i_\horizon(\state_\horizon) = \log \int e^{-\bar{c}^i(\state_\horizon,\tilde{\action}_\horizon^i)} d \tilde{\action}_\horizon^i.
\end{align}
Now, to minimize~\eqref{eq:cost-KL}, note that $\exp\left( V^i_\horizon(\state_\horizon)\right)$ is a constant for a given state $s_T$. Since the KL divergence is minimum when the two arguements are the same, the policy which minimizes~\eqref{eq:cost-KL} is 

\begin{align}
    \pi_\horizon^{i^*} (\action^i_{T}| \state_\horizon) = \frac{e^{-\bar{c}^i(\state_\horizon,\action_\horizon^i)}}{e^{V^i_\horizon(\state_\horizon)}}.
\end{align}

For every state $s_\horizon$, $V^i(s_\horizon)$ is indeed the cost-to-go (or value function) of agent $i$ at time $T$. We can use using dynamic programming to propagate this cost-to-go backwards in time and find the equilibrium policy for agent $i$ at time $\horizon-1$. At time $t-1$, we need to solve for
\begin{align}\label{eq:1}
\begin{split}
    \min_{\pi_{\horizon-1}^i} \,\, &\mathbb{E} \big[ c^i(\state_{\horizon-1},\action_{\horizon-1}^i,\bm{\action}_{ \horizon-1}^{{-i}^*}) \\ &\qquad\qquad -\mathcal{H}(\pi_{\horizon -1}^i(\cdot|\state_{\horizon -1})) + \mathbb{E}_{\state_\horizon}( V^i_\horizon(\state_\horizon))\big],
    \end{split}
\end{align}

\noindent where the first expectation is with respect to $\action_{ \horizon-1}^i \sim \pi^i_{\horizon -1}$ and $\bm{\action}_{\horizon -1}^{{-i}^*} \sim \bm{\pi}^{{-i}^*}_{\horizon -1}$. Again, we can separate the expectations in~\eqref{eq:1} as
\begin{align}\label{eq:2}
\begin{split}
    \min_{\pi_{\horizon-1}^i} \mathbb{E}_{\pi_{\horizon -1}^i} \Big( \mathbb{E}_{ \bm{\pi}^{{-i}^*}_{\horizon -1}} \left[ c^i(\state_{\horizon -1},\action_{\horizon -1}^i,\bm{\action}_{\horizon -1}^{{-i}^*}) +  \mathbb{E}_{\state_\horizon} V^i_\horizon(\state_\horizon) \right] \\ - \mathcal{H}(\pi_{\horizon -1}^i(\cdot|\state_{\horizon -1})) \Big).
    \end{split}
\end{align}
Then, using~\eqref{eq:Qbar-fun} we see that~\eqref{eq:2} can be rewritten as
\begin{align}\label{eq:3}
\begin{split}
    \min_{\pi_{\horizon-1}^i} \mathbb{E}_{ \pi_{\horizon -1}^i} \Big[ \bar{Q}^i_{T-1}(\state_{\horizon -1},\action_{\horizon -1}^i) - \mathcal{H}(\pi_{\horizon -1}^i(\cdot|\state_{\horizon -1})) \Big].
    \end{split}    
\end{align}
Thus, similar to minimizing~\eqref{eq:cost-to-KL}, to minimize~\eqref{eq:3}, the equilibrium policy for agent $i$ at time step $T-1$ is 
\begin{align}
    \pi^{i^*}_{\horizon -1}(\action_{\horizon -1}^i|\state_{\horizon -1}) = \frac{e^{-\bar{Q}_{\horizon -1}^i (\state_{\horizon -1},\action_{\horizon -1}^i)}}{e^{V_{\horizon -1}^i(\state_{\horizon -1})}},
    \end{align}
\noindent where the cost-to-go $V^i_{\horizon -1}(\state_{\horizon -1})$ is defined as
$$V_{\horizon-1}^i(\state_{\horizon -1}) = \log \int e^{-\bar{Q}_{\horizon -1}^i (\state_{\horizon -1},\tilde{\action}_{\horizon -1}^i)} d\tilde{\action}_{\horizon -1}^i.$$

Repeating this procedure, we can solve for the mixed-strategy Nash equilibrium policies backwards in time. Thus, we show via induction that mixed-strategy Nash equilibrium policies of the game $\tilde{G}=(\cc{S},\bm{\cc{A}}, f, g, \tilde{\bm{c}},\horizon)$ are in fact the ECE policies of the original game ${G}=(\cc{S},\bm{\cc{A}}, f, g, {\bm{c}},\horizon)$. To prove the reverse, i.e., every ECE of the game $G$ is a mixed strategy Nash equilibrium for $\tilde{G}=(\cc{S},\bm{\cc{A}}, f, g, \tilde{\bm{c}},T)$, we follow the same reasoning.  Starting from the final time step, one can show that ECE policies optimize the cost-to-go in the game $\tilde{G}$, and use induction to prove that ECE policies of $G$ are indeed Nash equilibria of $\tilde{G}$.
\end{proof}

Theorem~\ref{theorem:max-ent-equi} connects ECE to the Nash equilibria of a game between agents who aim to maximize the entropy of their policy while minimizing their accumulated cost. In other words, ECE is in fact an extension of the maximum entropy-framework in the single-agent setting to the setting of multiple interactive agents. It is well known that when it comes to learning from demonstrations in the single-agent scenarios, to capture the bounded rationality and noisiness of the demonstrator, the demonstrator is best modeled via the maximum-entropy framework~\cite{ziebart2008maximum,ziebart2010modeling}. Theorem~\ref{theorem:max-ent-equi} extends this notion to multi-agent games. We will use Theorem~\ref{theorem:max-ent-equi} in the remainder of this paper for computing the ECE policies in general dynamic games. 
\begin{remark}
The notion of ECE can incorporate a temperature parameter $\gamma^i$ for each agent such that $p(a^i|s_t) \propto \exp{\left(\frac{-\bar{Q}^i_{\bm{\pi}}(s_t,a^i_t)}{\gamma^i}\right)}$. This in turn will lead to an additional weight $\gamma^i$ on the entropy of each agent's policy in~\eqref{eq:equ-max-ent}. The temperature weight $\gamma^i$ captures each agent's rationality. The higher $\gamma^i$ is, the noisier the agent acts. If for each agent $i\in \ageset$, $\gamma^i \rightarrow \infty$, then in the limit, all agents act completely randomly with a uniform probability distribution over actions. On the other hand, when $\gamma^i \rightarrow 0$ for all agents, the Nash equilibrium policies are recovered. Therefore, the temperature coefficient $\gamma^i$ reflects the rationality of each agent. For simplicity, in this paper, we assume that $\gamma^i=1$ for all agents.
\end{remark}

\section{Solving for ECE Policies}\label{sec:eq-policies}
So far, we have defined ECE for capturing the interaction of boundedly rational and noisy agents. In this section, we give an algorithm for computing ECE policies for a given game $G=(\cc{S},\bm{\cc{A}}, f, g, \bm{c},\horizon)$ using Theorem~\ref{theorem:max-ent-equi}.  In the next section we use this ``forward'' solution to learn cost function parameters given trajectory demonstrations, i.e., we solve the inverse dynamic game.  
To solve the maximum entropy dynamic game, first, we prove that ECE policies can be computed in closed-form for the class of linear quadratic Gaussian games using a Riccati solution similar to the classic LQR~\cite{kalman1964linear} and LQGames~\cite{bacsar1998dynamic} solutions. Then, we extend this algorithm to general nonlinear dynamic games through iterative linear-quadratic approximations, similar to Differential Dynamic Programming~\cite{jacobson1968new}, iLQR~\cite{tassa2014control}, and iLQGames~\cite{fridovich2019efficient}. 


\subsection{Maximum Entropy Linear Quadratic Gaussian Games}
Consider a class of games $G=(\cc{S},\bm{\cc{A}}, f, g, \bm{c},\horizon)$ where the system dynamics are linear, and the system noise is normally distributed, i.e., the state update~\eqref{eq:dynamics} is of the following form
\begin{align}
\begin{split}
    \state_{t+1} &= A \state_t + \sum_{j\in \ageset} B^j \action^j_t + w_t \label{eq:dynamics-linear},\\
     w_t &\sim \mathcal{N}(0,I),
\end{split}
\end{align}
where $A \in \bb{R}^{n\times n}$, and $B^j \in \bb{R}^{m^j \times m^j}$ are known time-invariant matrices of appropriate dimensions, and $w_t$ is a zero-mean normally distributed random variable with covariance matrix being identity. Note that in general, the dynamics matrices can be time-variant, and the process noise can be any normal distribution. Here, for the ease of description, we assume system dynamics are linear time-invariant subject to process noise with covariance matrix being identity. Moreover, consider the class of cost functions where every agent $i \in \ageset$ minimizes a convex quadratic cost function for every time step $t \leq  \horizon$:
\begin{align}\label{eq:quadratic-stage-cost}
\begin{split}
    c^i(\state_t,\action^1_t,\cdots,\action^\agenum_t) =  \frac{1}{2}\big( \state_t^\transpose Q^i \state_t + {l^i}^\transpose s_t + \\ \qquad \sum_{j\in \ageset}  {\action_t^j}^\transpose R^{ij} {\action^j_t}\big),
\end{split}
\end{align}
where $Q^i \in \bb{R}^{n \times n}$ is a positive semi-definite matrix, and $l^i$ is a vector capturing the affine penalty of agent $i$ for the system state. Moreover for every agent $j \in \ageset, j\neq i$, $R^{ij} \in \bb{R}^{m^j \times m^j}$ is a positive semi-definite matrix while $R^{ii}$ is a positive definite matrix. Equations~\eqref{eq:dynamics-linear} and~\eqref{eq:quadratic-stage-cost} define the class of linear quadratic games where the dynamics are linear and the stage costs are quadratic in the states and actions.
\begin{definition}
A given game $G=(\cc{S},\bm{\cc{A}}, f, g, \bm{c},\horizon)$ is a linear quadratic Gaussian game if its dynamics are of the form~\eqref{eq:dynamics-linear}, and further, for each agent $i \in \ageset$, the cost function $c^i$ is of the form~\eqref{eq:quadratic-stage-cost}.
\end{definition}

 Note that for a linear quadratic Gaussian game $G$, in the resulting maximum entropy game $\tilde{G}$, every agent $i \in \ageset$ minimizes the following accumulated cost when we fix the policy of all the other agents to be $\bm{\pi}^{-i}$,
\begin{align}\label{eq:quadratic-cost-accumlated}
\begin{split}
    \min_{\pi^i} \, \mathbb{E}_{\pi^i} \mathbb{E}_{ \bm{\pi}^{-i}} & \sum_{k=1}^{\horizon} \frac{1}{2}\Bigg(  \state_k^\transpose Q^i \state_k + {l^i}^\transpose \state_k + \\ &\sum_{j\in \ageset}  {\action_k^j}^\transpose R^{ij} {\action^j_k}\Bigg) -  \sum_{k=1}^{T} \mathcal{H} (\pi^i(.|s_k).
\end{split}
\end{align}

Now that we have defined the class of linear-quadratic games, using~\eqref{eq:quadratic-cost-accumlated} and Theorem~\ref{theorem:max-ent-equi}, we will prove how ECE policies can be obtained in closed form for this class of games.

\begin{theorem}\label{theorem:LQ-game} Consider a linear quadratic Gaussian game $G=(\cc{S},\bm{\cc{A}}, f, g, \bm{c},\horizon)$. For every agent $i \in \ageset$, at every time step $t < \horizon$, the ECE policy $\pi^{i^*}_t$ is a normal distribution $\pi_t^{i^*} \sim \cc{N}(\mu^i_t,\Sigma^i_t)$ where the mean $\mu^i_t$ and the covariance $\Sigma^i_t$ are computed by
\begin{align}
    \mu_t^i &= -P_t^i \state_t - \alpha_t^i, \label{eq:mu}\\
    \Sigma_t^i &= (R^{ii} + {B^i}^\transpose Z_{t+1}^i B^i)^{-1}. \label{eq:sigma}
\end{align}
The matrices $P_t^i$ and the vectors $\alpha_t^i$ satisfy the following sets of linear equations
\begin{align}
\begin{split}
\left[R^{ii} + {B^i}^\transpose Z_{t+1}^i B^i\right] P_t^i + {B^i}^\transpose Z_{t+1}^i \sum_{j \in \ageset, j\neq i} & B^j P_t^j = \\ &{B^i}^\transpose Z_{t+1}^i A,    
\end{split} \label{eq:P-recur}\\
\begin{split}
\left[R^{ii} + {B^i}^\transpose Z_{t+1}^i B^i\right] \alpha_t^i + {B^i}^\transpose Z_{t+1}^i \sum_{j \in \ageset, j\neq i} & B^j \alpha_t^j = \\ &{B^i}^\transpose \xi_{t+1}^i,    
\end{split}\label{eq:alpha-recur}
\end{align}
where $Z_t^i$, $\xi_t^i$ are recursively computed from the following
\begin{align}
    Z_t^i &= F_t^{\mathsf{T}} Z_{t+1}^i F_t + \sum_{j\in \ageset} {P_t^j}^\transpose R^{ij} P_t^j + Q^i \label{eq:Z-update} \\
    \xi_t^i &= F_t^\transpose (\xi_{t+1}^i + Z_{t+1}^i \beta_t) + \sum_{j\in \ageset } {P_t^j}^\transpose R^{ij}\alpha_t^j , \label{eq:xi-update}
\end{align}
where 
\begin{align}
    F_t &= A_t - \sum_{j\in\ageset} B^j P_t^j \label{eq:F-update} \\ \beta_t &= -\sum_{j\in\ageset} B^j \alpha^j_t. \label{eq:beta_update}
\end{align}
The terminal conditions for Equations~\eqref{eq:Z-update} and~\eqref{eq:xi-update} are 
\begin{align}\label{eq:terminal-cost-init}
\xi^i_\horizon = l^i, \quad Z_{\horizon} = Q^i.
\end{align}
\end{theorem}

\begin{proof}
The proof can be found in Appendix~\eqref{appendix-proof-thereom2}.
\end{proof}

Note that Theorem~\ref{theorem:LQ-game} is in fact an extension of Maximum Entropy LQR derived in~\cite{ziebart2010modeling} to the multi-agent setting. Equations~\eqref{eq:P-recur},~\eqref{eq:alpha-recur},~\eqref{eq:Z-update},~\eqref{eq:xi-update} are the extensions of the Riccati backward recursion to the multi-agent setting. Indeed, these equations are similar to the backward recursions for deterministic LQ games derived in~\cite{bacsar1998dynamic}. Theorem~\ref{theorem:LQ-game} suggests that for Maximum Entropy linear quadratic games, the optimal policy is obtained by a normal distribution whose mean is found by solving LQ games. Then, the variance of the normal policies at every time step is found via~\eqref{eq:sigma}.


\subsection{General Nonlinear Games}
Theorem~\ref{theorem:LQ-game} provides a closed-loop expression for computing ECE policies of linear quadratic Gaussian games over a finite horizon of time $\horizon$. Inspired by~\cite{fridovich2019efficient,wang_game-theoretic_2020}, we will leverage this result for approximating the ECE trajectories in general games with nonlinear dynamics and non-quadratic cost functions. We assume that each agent $i$ can minimize a possibly nonlinear cost function of the form
\begin{equation}
\label{eq:additive_cost}
    c^{i}(\state_t,\bm{\action}_t) = 
        v_{\state,t}^{i}(\state_{t}) + \sum_{j\in \ageset} a_t^{j^\transpose} R^{ij} a_t^j,
\end{equation}
where $v_{\state,t}^{i}(\state_{t})$ is a nonlinear state cost. Note that in general, we could have non-quadratic nonlinear costs on the actions as well; however, since~\eqref{eq:additive_cost} captures a wide range of applications, we consider cost functions of the form~\eqref{eq:additive_cost}.

To approximate ECE policies, we propose an iterative algorithm where we start by a nominal trajectory. Then, we linearize the dynamics, find a quadratic approximation of the cost function for every agent, and solve the resulting Maximum Entropy linear quadratic Gaussian game. We then update our reference trajectory and repeat this process until convergence (see Algorithm~\ref{alg:ERE-trajectory}). Consequently, our algorithm shares a similar structure with DDP~\cite{jacobson1968new,tassa2014control}
and iLQR methods~\cite{li2004iterative}. More precisely, we start with a nominal sequence of the gain matrices and offset vectors $\{P_{t}^{i}, \alpha_{t}^{i}\}$ for every time step $t\in \horizon,$ and every agent $i \in \ageset$.  If such a nominal sequence of control matrices is not available, a trivial initialization is initializing all matrices to zeros. We choose the mean value found by~\eqref{eq:mu} as our reference controller. Then, at every iteration, a sequence of states and actions $\eta = \{\bar{\state}, \bar{\bm{\action}}^1= ({\mu}^{1}_1,\cdots,{\mu}^1_\horizon),\cdots, \bar{\bm{\action}}^\agenum=({\mu}_1^{\agenum},\cdots, \mu_\horizon^\agenum)\}$ is found by forward simulating the system dynamics using the nominal control inputs (which are chosen to be $\mu^i_t$ for each agent $i$ at time $t$). We then linearize the system dynamics~\eqref{eq:dynamics} around this nominal trajectory to obtain
\begin{equation}
\label{eq:linear_approx}
\delta \state_{t+1} \approx A_{t} \delta \state_{t} + \sum_{j\in \ageset }B_{t}^{j}\delta \action_{t}^{j},
\end{equation}
where $\delta s_{t} = s_{t} - \bar{s}_{t}$, $\delta a_{t}^{i} = a_{t}^{i} - \bar{a}^{i}_{t}$, and $A_{t} = D_{\state}f(\cdot)$ and $B_{t}^{j} = D_{\action_{t}^{j}}f(\cdot)$ are the Jacobians of the system dynamics~\eqref{eq:dynamics} with respect to the state $\state_{t}$, and actions $\action_{t}^{j}$, respectively. We further acquire a quadratic approximation of the cost function for each agent. For each agent $i \in \ageset$, we let
\begin{equation}
\label{eq:ileq_quad_approx}
    c^{i}(\bar{\state}_{t}+\delta \state_{t},\bar{\bm{\action}}+\delta \bm{\action}) \approx v_{\state,t}^{i}(\bar{\state}_{t}) + \frac{1}{2}\delta \state_{t}^{\transpose} H_{t}^{i} \delta \state_{t} + l_{t}^{i\transpose} \delta \state_{t},
\end{equation}
denote such an approximation
where $H_{t}^{i} = D_{\state_{t}\state_{t}}v_{\state,t}^{i}(\cdot)$ and $l_{t}^{i}= D_{\state_{t}}v_{\state,t}^{i}(\cdot)$ are the Hessian and the gradient of the cost function $v_{\state,t}^{i}(\cdot)$ with respect to $\state_{t}$. Note that our formulation only considers nonlinear costs on state variables, and the dependence on agents' actions is quadratic. For a more general case, where the cost function is nonlinear in control actions too, a similar approximation could be used to derive the quadratic terms and linear terms in $\action_{t}^{i}$. Note that all the approximations $A_{t}, B_{t}^{j}, H_{t}^{j}, l_{t}^{j}$ are evaluated at $\eta$.

For the linearized system dynamics~\eqref{eq:linear_approx} and quadratized cost~\eqref{eq:ileq_quad_approx}, we reach a new approximated linear quadratic Gaussian game with new variable sequences $\delta \state, \delta \bm{\action}^{1}, \cdots, \delta \bm{\action}^{\agenum}$. These approximations result in a new game that can be solved using Theorem~\ref{theorem:LQ-game}. Once the approximated game is solved, we obtain a new sequence of mean control actions
\begin{equation}
\label{eq:ileq_delta_u}
\begin{split}
    \{\bar{\action}_{t}^{i} + \delta \action_{t}^{i^*}, t = 0,\cdots, \horizon-1\},
\end{split}
\end{equation}
where $\delta \action_{t}^{i^*}$ is the mean value of the action distribution found by solving for the ECE policies of the approximated linear quadratic Gaussian game. 
A new $\bar{\state}_{t}$ is attained from the forward simulation of the original system dynamics~\eqref{eq:dynamics} using the newly obtained nominal control actions.
We repeat the above process until convergence, i.e., the deviation of the new state trajectory from the state trajectory in the previous iteration lies within a desired tolerance. Note that in our iterative process, we always forward simulate the mean value of the distribution of agents' action. Once the mean trajectories converge, we sample ECE policies by sampling control actions from~\eqref{eq:mu} and~\eqref{eq:sigma} computed around the converged mean trajectory. 

\begin{remark}
Applying $a_{t}^{i}$ directly from (\ref{eq:ileq_delta_u}) may lead to non-convergence since the resulting trajectory could deviate too much from the original non-linear system which we approximated around $\eta_{t}$. As in other iterative linear quadratic methods \cite{yakowitz2012algorithms, van2012motion, fridovich2019efficient}, we only take a small step in the proposed direction. At each iteration, rather than~\eqref{eq:ileq_delta_u},  we apply the following control input
\begin{equation}
\label{eq:ileq_delta_u_withstep}
     \bar{a}_{t}^{i} - P_{t}^{i}\delta s_{t} - \epsilon \alpha_{t}^{i},
\end{equation}
where $\epsilon$ is the step size for improving our control strategy. Initially, we set $\epsilon = 1$. Drawing inspiration from line search method in optimization problems, we decrease $\epsilon$ by half until the new trajectory's deviation from the nominal trajectory is within a threshold. 
\end{remark}

\begin{algorithm}
    \caption{Approximating ECE Trajectories}
    \label{alg:ERE-trajectory}
    \begin{algorithmic}[1]
        \State \textbf{Inputs}
        \State system dynamics (\ref{eq:dynamics}), agents' cost functions (\ref{eq:additive_cost})
        \State \textbf{Initialization}
        \State initialize the control policy using $P_{t}^{i} = 0$, and $\alpha_{t}^{i} = 0$, $\forall i \in \ageset$
        \State forward simulation and obtain $(\bar{\state}_{0}, \bar{\state}_{1}, \cdots, \bar{s}_{\horizon})$, $\bar{\bm{a}}^i$, $\cdots, \bar{\bm{a}}^\agenum$
        \While{not converged}
            \State linear approximation of (\ref{eq:dynamics}) 
            \State quadratic approximation of (\ref{eq:additive_cost})
            \State solve the backward recursion with (\ref{eq:mu}-\ref{eq:terminal-cost-init})
            \State forward simulation using $\mu^j$'s and obtain the new trajectories
         \EndWhile
        \State \Return $P_{t}^{i}$, $\alpha_{t}^{i}$, and $Z_{t+1}^i$
    \end{algorithmic}
\end{algorithm}

\section{Learning Agents' costs}\label{sec:ma-irl}
So far, we have discussed the concept of noisy equilibrium, ECE, that we adopt and how to solve for approximate equilibrium policies under this model. Now, we are ready to discuss how to learn agents' cost functions $c^i$ from a set of interaction demonstrations from humans or other experts. The main intuition behind our algorithm is similar to the single-agent Maximum Entropy IRL~\cite{ziebart2008maximum}. The common assumption in single-agent IRL is that the agent's cost function is parameterized as a linear combination of a set of features, and the weight of features is learned such that the expectation of the features under the learned cost function matches the empirical feature means under the demonstrations. This is achieved through an iterative algorithm where at every iteration, the difference between the feature expectations under the demonstrations and under the learned cost functions is utilized for updating the cost parameters.

We extend such an iterative cost learning framework to the multi-agent game setting. We assume that agents' dynamics are known and each agent's cost function $c^i(\cdot)$ is parameterized as a weighted sum of features, i.e., for each agent $i \in \ageset$, we have
\begin{align}\label{eq:cost-linear-vector}
    c^i(\state_t,\bm{\action}_t)= w^{i^\transpose} \phi^i (\state_t, \bm{\action}_t),
\end{align}
where $w^i$ is the vector of weight parameters for agent $i$ and $\phi^i(\state_t, \bm{\action}_t)$ is the vector of features for agent $i$. Note that for each agent $i$, the vector of features $\phi^i(\state_t, \bm{\action}_t)$ depends on the entire state vector and the actions of all agents. We define $\bm{w}=(w^1,\cdots, w^\agenum)$ to be the collection of cost weights for all agents. 

 Our goal is to find the weight parameters $\bm{w}$ such that the feature expectation  under the learned cost weights matches the empirical feature mean under the demonstrations. Let $\mathcal{D}_{\bm{w}}$ be the probability distribution over equilibrium trajectories induced by the cost parameters $\bm{w}$. We further let $\bar{\mathcal{D}}$ be the emirical distrubution of equilibrium trajectories in the interaction demonstrations. Moreover, we define a trajectory $(s,\bm{a})=(s_t,\bm{a}_t)_{t=1}^T$ to be the vector of agents' states and actions over a trajectorty of length $T$. 
 We assume that agents maintain an entropic cost equilibrium in their demonstrations. 

With a slight abuse of notation, for each agent $i$, we denote the empirical mean of the features under the interaction demonstrations by  $\mathbb{E}_{(s,\bm{a})\sim \bar{\mathcal{D}}} \, \phi^i(\state, \bm{\action})$. Likewise, we let $\mathbb{E}_{(s,\bm{a})\sim \mathcal{D}_{\bm{w}}} \, \phi^i(\state, \bm{\action})$ be the expected value of the features under the equilibrium trajectories induced by the weight vector $\bm{w}$.

We want to find weight parameters $\bm{w}$ to induce a probability distribution $\mathcal{D}$ over equilibrium trajectories such that feature expectation under the learned cost weights matches the empirical feature mean under demonstrations, i.e., $\mathbb{E}_{(s,\bm{a})\sim \bar{\mathcal{D}}} \, \phi^i(\state, \bm{\action})=\mathbb{E}_{(s,\bm{a})\sim \mathcal{D}_{\bm{w}}} \, \phi^i(\state, \bm{\action})$ for all agents $i\in \ageset$. To this end, we propose an iterative algorithm. We initialize the cost weights $\bm{w}$ for all agents. At each iteration of the algorithm, we iterate over all the agents. For agent $i$, we compute the difference between the feature expectations under the demonstrations and the current model
$\mathbb{E}_{(s,\bm{a})\sim \bar{\mathcal{D}}} \, \phi^i(\state, \bm{\action})-\mathbb{E}_{(s,\bm{a})\sim \mathcal{D}_{\bm{w}}} \, \phi^i(\state, \bm{\action}) $ and use this difference for updating the weight parameters
\begin{align}\label{eq:cost-update}
     w^i \leftarrow w^i -  \gamma \left( \mathbb{E}_{(s,\bm{a})\sim \bar{\mathcal{D}}} \, \phi^i(\state, \bm{\action})-\mathbb{E}_{(s,\bm{a})\sim \mathcal{D}_{\bm{w}}} \, \phi^i(\state, \bm{\action}) \right),
\end{align}
where $\gamma$ is the learning rate. Note that once $w^i$ is updated, the entire vector of cost weights $\bm{w}$ is in fact updated. In the single-agent setting where an agent is nosily optimizing a cost function, the update rule~\eqref{eq:cost-update} is the gradient of the likelihood of trajectories under weight parameters $\bm{w}$. In other words, in the single-agent setting,~\eqref{eq:cost-update} solves a maximum likelihood problem for finding weight parameters $\bm{w}$ that maximize the likelihood of demonstrations. In the multi-agent setting, if we fix the policy and cost parameters of all agents except agent $i$, the equilibrium policy of agent $i$ is the optimal policy which minimizes agent $i$'s cost, i.e., the problem reduces to a single-agent setting where agent $i$ noisly minimizes its cost. Thus, if we fix all agents except agent $i$, similar to the single-agent setting, we can interpret~\eqref{eq:cost-update} as a gradient-descent update rule. Once the cost parameters of agent $i$ are updated, we compute the feature expectations under the new set of cost parameters, and update the cost parameters for the next agent. We iterate over all $i$, and repeat this process until convergence.  This can be interpreted as a block coordinate descent method for solving the coupled parameter estimation problems for the multi-agent game. 

Algorithm~\ref{alg:MA-IRL} summarizes the algorithm steps. Note that in a multi-agent setting, once the cost parameters of one agent are updated, we need to recompute the feature expectation for updating the next agent's cost, since agents' cost parameters and feature expectations are interdependent, i.e., a change in one agent's cost parameters affects the expectation of features for other agents. This is due to the coupling between agents at entropic cost equilibrium.

\begin{algorithm}
    \caption{Multi-Agent Inverse Reinforcemenr Leaning}
    \label{alg:MA-IRL}
    \begin{algorithmic}[1]
        \State \textbf{Inputs}
        \State system dynamics (\ref{eq:dynamics}), agents' cost features $\phi^i(\cdot)$, and the set of interaction demonstrations. 
        \State \textbf{Initialization}
        \State for each agent $i \in \ageset$, compute the empirical feature mean under demonstrations $\mathbb{E}_{(s,\bm{a})\sim \bar{\mathcal{D}}} \, \phi^i(\state, \bm{\action})$
        \State initialize the cost weights $\bm{w}$
        \While{not converged}
        \For{each agent $i \in \ageset$}
            \State Compute feature expectations $\mathbb{E}_{(s,\bm{a})\sim \mathcal{D}_{\bm{w}}} \, \phi^i(\state, \bm{\action}) $ 
            \State Compute the difference $$\mathbb{E}_{(s,\bm{a})\sim \bar{\mathcal{D}}} \, \phi^i(\state, \bm{\action})-\mathbb{E}_{(s,\bm{a})\sim \mathcal{D}_{\bm{w}}} \, \phi^i(\state, \bm{\action}) $$ 
            \State{Update cost weight for agent $i$:
            $$ w^i \leftarrow w^i -  \gamma \left( \mathbb{E}_{(s,\bm{a})\sim \bar{\mathcal{D}}} \, \phi^i(\state, \bm{\action})-\mathbb{E}_{(s,\bm{a})\sim \mathcal{D}_{\bm{w}}} \, \phi^i(\state, \bm{\action}) \right)  $$} 
        \EndFor
         \EndWhile
        \State \Return cost parameters $w^i$ for all agents $i \in \ageset$
    \end{algorithmic}
\end{algorithm}

Each iteration of Algorithm~\ref{alg:MA-IRL} at line 8 requires computing agents' feature expectation $\mathbb{E}_{(s,\bm{a})\sim \mathcal{D}_{\bm{w}}} \, \phi^i(\state, \bm{\action})$ under ECE policies. However, computing $\mathbb{E}_{(s,\bm{a})\sim \mathcal{D}_{\bm{w}}} \, \phi^i(\state, \bm{\action})$ for general nonlinear games is not tractable. We propose to approximate this by sampling ECE policies through Algorithm~\ref{alg:ERE-trajectory} under the current cost parameter estimates. In each iteration, we sample $p$ ECE equilibrium trajectories $(s_1,\bm{a}_1), \cdots, (s_p,\bm{a}_p)$ using Algorithm~\ref{alg:ERE-trajectory}, and approximate these through the following
\begin{align}
\mathbb{E}_{(s,\bm{a})\sim \mathcal{D}_{\bm{w}}} \, \phi^i(\state, \bm{\action}) \approx \frac{1}{p}  \sum_{j=1}^p \phi^i(s_j,\bm{a}_j). 
\end{align}
Thus, with known system dynamics, Algorithm~\ref{alg:ERE-trajectory} lets us compare equilibrium trajectories under the learned cost parameters with the demonstration trajectories and update the cost parameters if needed. Algorithm~\ref{alg:exp-feature} summarizes the computation of feature expectations.

\begin{algorithm}
    \caption{Approximating Feature Expectations}
    \label{alg:exp-feature}
    \begin{algorithmic}[1]
        \State \textbf{Inputs}
        \State system dynamics (\ref{eq:dynamics}), agents' cost features $\phi^i(\cdot)$, agents' cost wights $w^i$, number of samples $p$ 
        \For{$ 1\leq j \leq p$}
            \State Sample an ECE trajectory $(s_j,\bm{a}_j) $
        \EndFor
        \State \Return $\frac{1}{p}  \sum_{j} \phi^i(s_j, \bm{a}_j)$
    \end{algorithmic}
\end{algorithm}

\section{Experiments}\label{sec:experiments}
In this section, we evaluate the performance of the MA-IRL algorithm in two different scenarios. In the first scenario, we consider a motion planning task with multiple agents that navigate to goal locations while avoiding collisions with each other. In the second scenario, we use the INTERACTION dataset \cite{interactiondataset} which contains trajectories in interactive traffic scenes, such as intersection, roundabouts, and highway merging, collected from different locations in different countries.

We compare our algorithm with a baseline method, continuous inverse optimal control (CIOC) \cite{levine2012continuous, sadigh2016planning}, which is an IOC algorithm for large, continuous domains and does not assume any feedback interaction among the observed agents. Additionally, for the INTERACTION dataset, we also compare with the intelligent driver model (IDM) \cite{treiber2000congested, Bhattacharyya2020} as our reference model.  IDM is a widely used expert-designed model to simulate traffic flow that is known to yield accurate predictions of drivers' trajectories.

\subsection{Motion Planning with Collision Avoidance}
We first consider a synthetic environment where 2 or 3 agents move to their goal locations and try to avoid collisions with each other. Demonstrations for all scenarios are generated by solving an approximate entropic cost equilibrium as described in Sec.~\ref{sec:eq-policies}. Fig.~\ref{fig:synthetic_demonstration} illustrates the demonstration trajectories for 2 and 3 agents scenarios. Each agent wants to minimize a cost function which is linear in the set of features: tracking a reference trajectory, penalizing close proximity to other agents, and control effort. Plots of different features are shown in Fig.~\ref{fig:synthetic_feature}. The reference trajectory is a straight line between initial location and goal location. Each agent has different weightings on the features, which leads to different interactive behaviors as shown in Fig.~\ref{fig:synthetic_demonstration}.
\begin{figure}[!t]
\centering
\begin{subfigure}{0.45\columnwidth}
  \centering
  \includegraphics[width=\columnwidth]{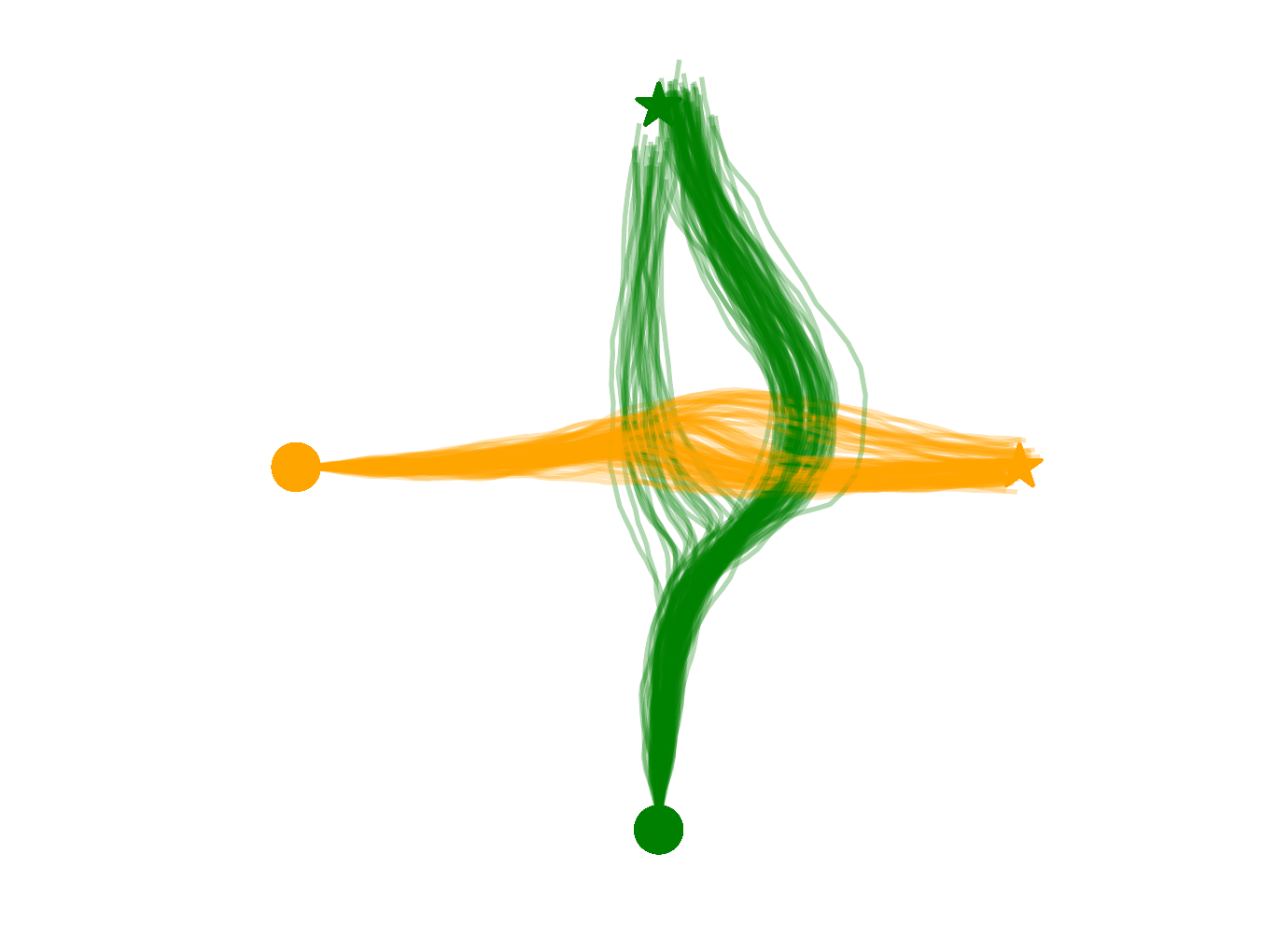}
  \caption{Two-player case}
  \label{fig:two_player_demonstration}
\end{subfigure}
\hfill
\begin{subfigure}{0.45\columnwidth}
  \centering
  \includegraphics[width=\columnwidth]{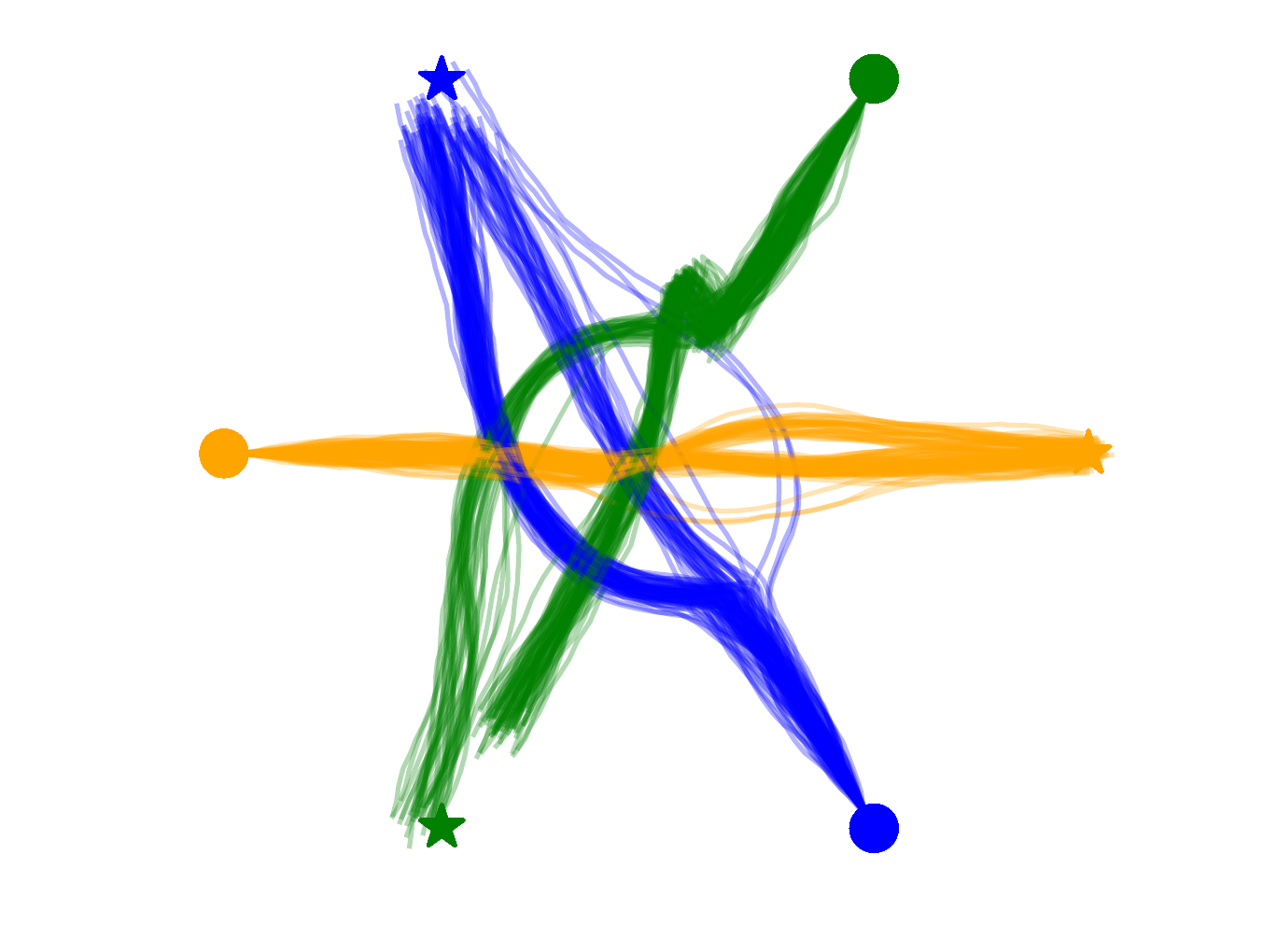}
  \caption{Three-player case}
  \label{fig:three_player_demonstration}
\end{subfigure}
\caption{Demonstrations for two-player and three-player collision avoidance motion planning. Trajectories of different players are shown in green, yellow, and blue, respectively. Players have different cost functions. Due to the stochastic policy, we observe different modes of interactions.}
\label{fig:synthetic_demonstration}
\end{figure}

\begin{figure}[!t]
\centering
\begin{subfigure}{0.45\columnwidth}
  \centering
  \includegraphics[width=\columnwidth]{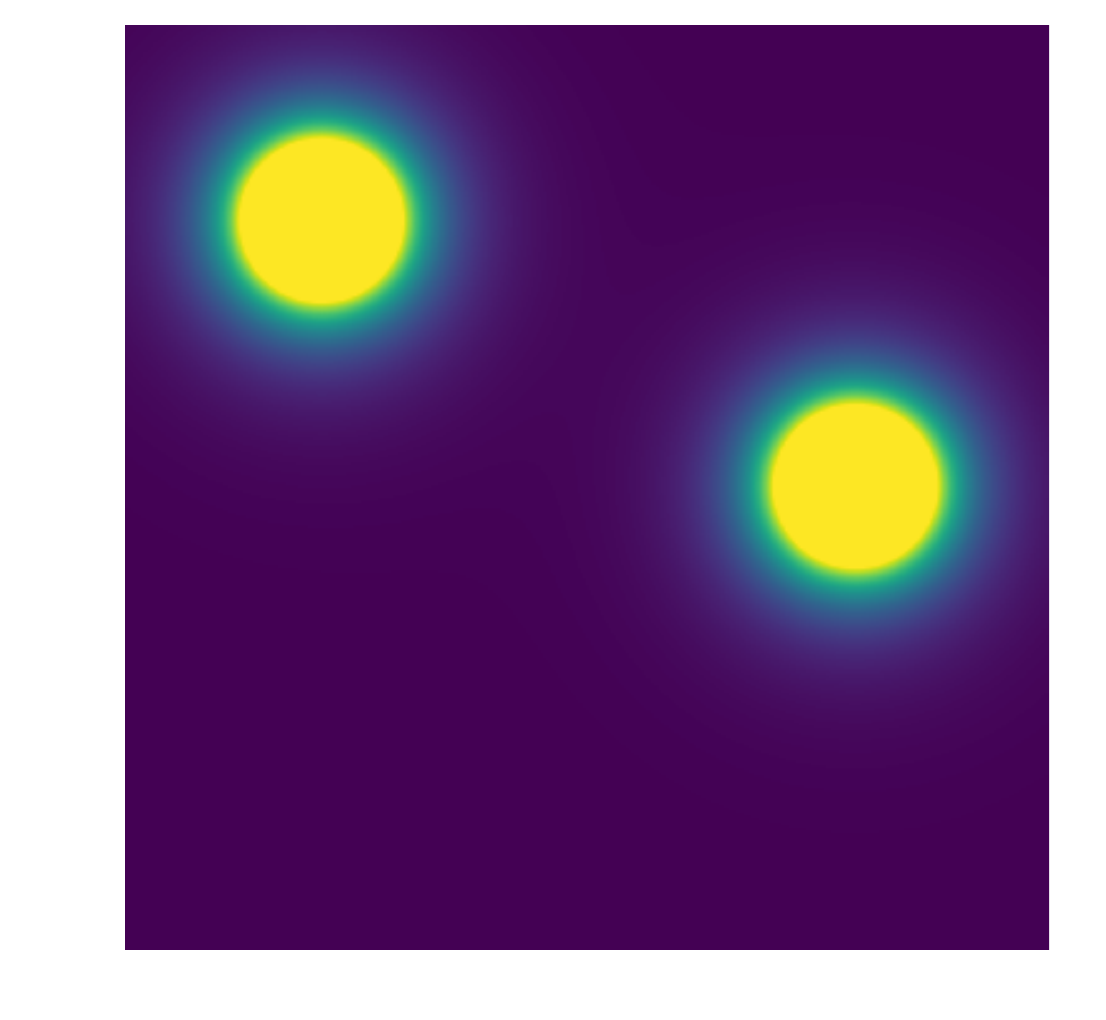}
  \caption{Collision avoidance feature}
  \label{fig:collision_feature}
\end{subfigure}
\hfill
\begin{subfigure}{0.45\columnwidth}
  \centering
  \includegraphics[width=\columnwidth]{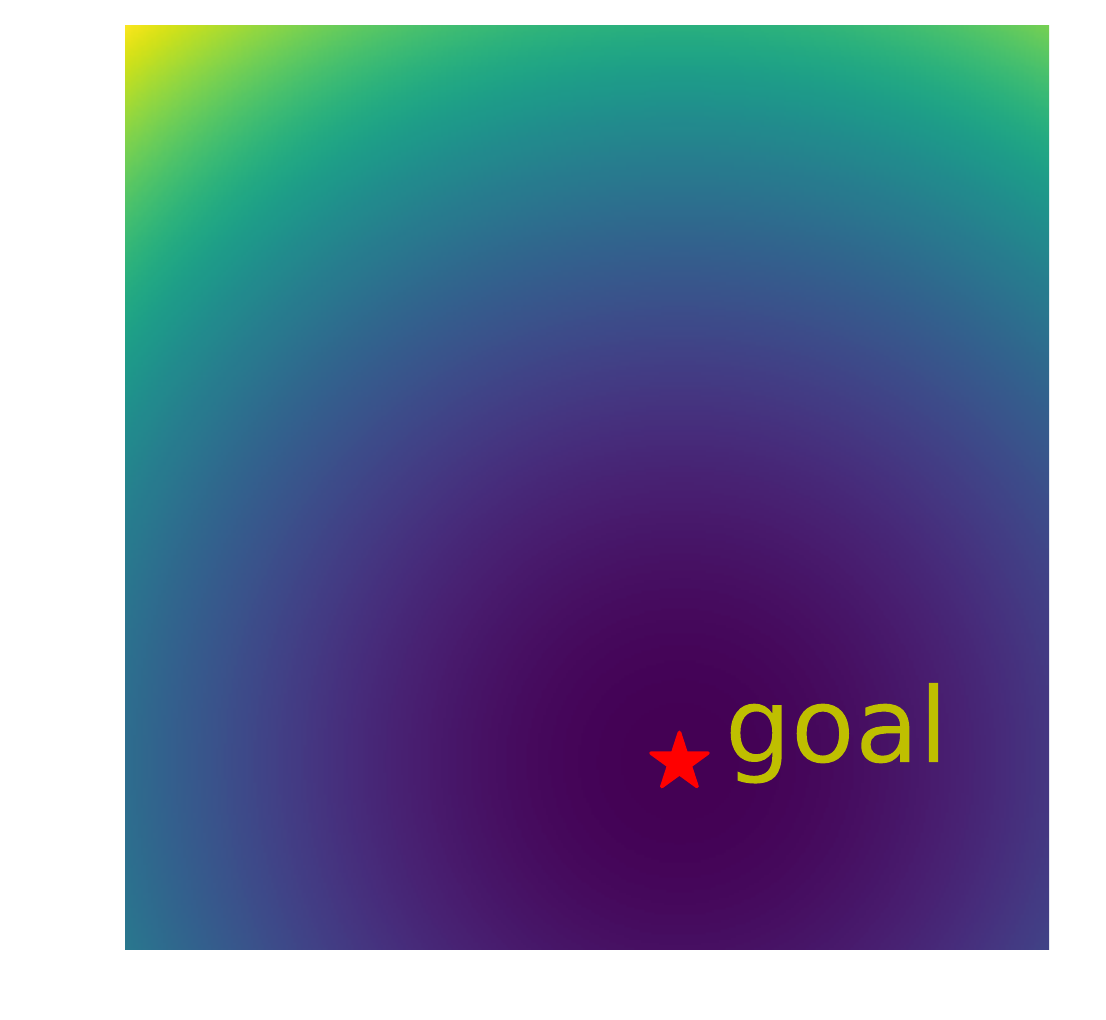}
  \caption{Reach goal location feature.}
  \label{fig:goal_feature}
\end{subfigure}
\caption{Cost features for motion planning with collision avoidance task. Costs are high close to other agents and away from the goal location.}
\label{fig:synthetic_feature}
\end{figure}

The demonstration dataset contains $200$ trajectories. We use MA-IRL to learn the cost function coefficients for all agents jointly. To apply CIOC method, we learn the cost coefficients of each agent individually and treat all other agents as obstacles. In our implementation of the CIOC method, we use a generic optimization solver. The computation time scales cubically with the planning horizon. For computational tractability, we partition the demonstration trajectories into sections of shorter trajectories.

We evaluate the performance of both algorithms in test scenarios where the initial configurations are generated randomly and are different from the demonstration data set. To evaluate the performance, we first solve the ECE solutions with the learned cost functions using both our MA-IRL approach and the CIOC approach, as well as  with the true cost functions in test scenarios. Then, we simulate the stochastic equilibrium policies for $200$ trials in each case and compute the feature expectations. The cost of each simulation for all agents are then computed with the true cost function. If the feature distribution from the learned cost functions is similar to the feature distribution in the demonstration data set, it means our learned costs are close to the true costs. The simulation results are shown in Table~\ref{tbl:synthetic_two_player_feature_kl_divergence} and \ref{tbl:synthetic_three_player_feature_kl_divergence}. 
\begin{table}[h]
  \centering
  \resizebox{\columnwidth}{!}{
    \begin{tabular}{lccccccccc}
    \toprule
    \multicolumn{2}{c}{\bf KL DIV} & & \multicolumn{3}{c}{\bf agent 1} & &  \multicolumn{3}{c}{\bf agent 2}  \\
    \cmidrule{4-6} \cmidrule{8-10} 
    & & & tracking & control & obstacle  & &  tracking & control & obstacle \\
    \midrule
    \multirow{2}{*}{\bf task 1} 
    & MA-IRL (ours)& & \textbf{0.054} & \textbf{0.024} & \textbf{0.044} & & \textbf{0.134} & \textbf{0.059} & \textbf{0.044} \\
    & CIOC         & & 1.701      & 0.165         & 1.766          & & 0.836          & 0.065          & 1.766 \\
    \midrule
    \multirow{2}{*}{\bf task 2} 
    & MA-IRL (ours) & & \textbf{0.025}& \textbf{0.080}          & \textbf{0.061} & & \textbf{0.034} & \textbf{0.053} & \textbf{0.061} \\
    & CIOC          & &  1.904        & 0.146 & 1.086          & & 0.152          & 0.144          & 1.086 \\
    \bottomrule
    \end{tabular}
  } 
  \caption{Two-agent collision avoidance simulations. KL-divergence of feature distribution between demonstrations and simulations from learned costs. Smaller values are better.}
  \label{tbl:synthetic_two_player_feature_kl_divergence}
\end{table}


We compute the Kullback–Leibler divergence, which captures the difference between two distributions. Lower values means the distributions are closer.
In the two different test tasks, MA-IRL algorithm achieves closer similarity to true feature distribution, which means the MA-IRL learned cost functions generalize to different tasks better than CIOC.

\begin{table*}[t]
  \centering
  \resizebox{\textwidth}{!}{
    \begin{tabular}{lccccccccccccc}
    \toprule
    \multicolumn{2}{c}{\bf KL DIV} & & \multicolumn{3}{c}{\bf agent 1} & &  \multicolumn{3}{c}{\bf agent 2} & & \multicolumn{3}{c}{\bf agent 3}  \\
    \cmidrule{4-6} \cmidrule{8-10} \cmidrule{12-14} 
    & & & tracking & control & obstacle  & &  tracking & control & obstacle & &  tracking & control & obstacle \\
    \midrule
    \multirow{2}{*}{\bf task 1} 
    & MA-IRL (ours) 
    & & \textbf{0.031} & 0.063          & \textbf{0.167} & & \textbf{0.074} & 0.144          & \textbf{0.052}
    & & \textbf{0.050 }& 0.066 & \textbf{0.127} \\
    & CIOC   
    & &  2.583         & \textbf{0.054} & 0.939          & & 1.083          & \textbf{0.039} & 0.489 
    & & 0.082 & \textbf{0.022} & 0.629 \\
    \midrule
    \multirow{2}{*}{\bf task 2} & MA-IRL (ours) 
    & &  \textbf{0.096} & \textbf{0.019} & \textbf{0.116} & & \textbf{0.054} & \textbf{0.019} & \textbf{0.148} & & \textbf{0.030} & \textbf{0.108} & \textbf{0.080} \\
                                & CIOC   
    & &  1.104 & 0.073 & 0.701 & & 1.685 & 0.032 & 0.780 & & 2.198 & 0.171 & 0.487 \\
    \bottomrule
    \end{tabular}
  } 
  \caption{Three-agent collision avoidance simulations. KL-divergence of feature distribution between demonstrations and simulations from learned costs. Smaller values are better.}
  \label{tbl:synthetic_three_player_feature_kl_divergence}
\end{table*}
\begin{table}[t]
\centering
\begin{tabular}{@{}c c c c c@{}}
\toprule
\multicolumn{2}{c}{avg. dist. to goal} &  true  & MA-IRL   &  CIOC    \\ 
\midrule
\multirow{2}{*}{\bf Task 1} & agent 1  & 0.205$\pm$0.103 & 0.229$\pm$0.117 & 0.144$\pm$0.083  \\ 
                            & agent 2  & 0.319$\pm$0.153 & 0.323$\pm$0.160 & 0.248$\pm$0.101  \\
\midrule 
\multirow{2}{*}{\bf Task 2} & agent 1  & 0.720$\pm$0.702 & 0.864$\pm$0.737 & 0.381$\pm$0.428\\ 
                            & agent 2  & 2.079$\pm$1.880 & 1.797$\pm$1.953 & 1.347$\pm$1.624  \\           
\bottomrule
\end{tabular}
\caption{Task-relevant statistics for two player collision avoidance case. Distances to goal location at the end of simulations are shown. The statistics of the MA-IRL approach are closer to the true cost samples.}
\label{tbl:two_player_task_statistics}
\end{table}


In addition, we follow prior work and also compute task-relevant statistics to evaluate how much the recovered cost and corresponding stochastic policy resembles the demonstrations. We compute the distance to goal location at the end of simulation ($6$ seconds) and the results are shown in Table.~\ref{tbl:two_player_task_statistics}. We can see that MA-IRL better resembles the true cost functions.

\subsection{INTERACTION dataset}
We further apply the MA-IRL algorithm on a real world traffic dataset, the INTERACTION dataset~\cite{interactiondataset}, which has highly interactive driving scenarios collected from different countries. Scenarios such as a roundabout, highway merging, and intersections are included.
\subsubsection{Data Processing}
We pick a highway merging scenario as it involves the vehicles negotiating their merging actions, leading to challenging game-theoretic planning for all vehicles. Among all the merging trajectories, each data point consists of three vehicles: a merging vehicle whose trajectory starts from the onramp and ends on a main highway lane, a leader vehicle that is ahead of the merging vehicle, and a follower vehicle that is behind the merging vehicle in the highway. Leader vehicles are treated as obstacles which do not interact with the other two vehicles, and we assume that the merging vehicle and the follower vehicle follow an ECE solution. We discard the trajectories that are too long or too short (longer than 9 seconds or shorter than 5 seconds) and end up with 87 trajectories in total. All $87$ data points are recorded in the same highway onramp in China from different recordings. The average time to finish merging is $6.3$ seconds in the data set. Fig.~\ref{fig:interaction_example} shows an example snapshot, where car $8$ is merging onto highway after car $13$.
\begin{figure*}[h]
    \centering
    \includegraphics[width=0.8\linewidth]{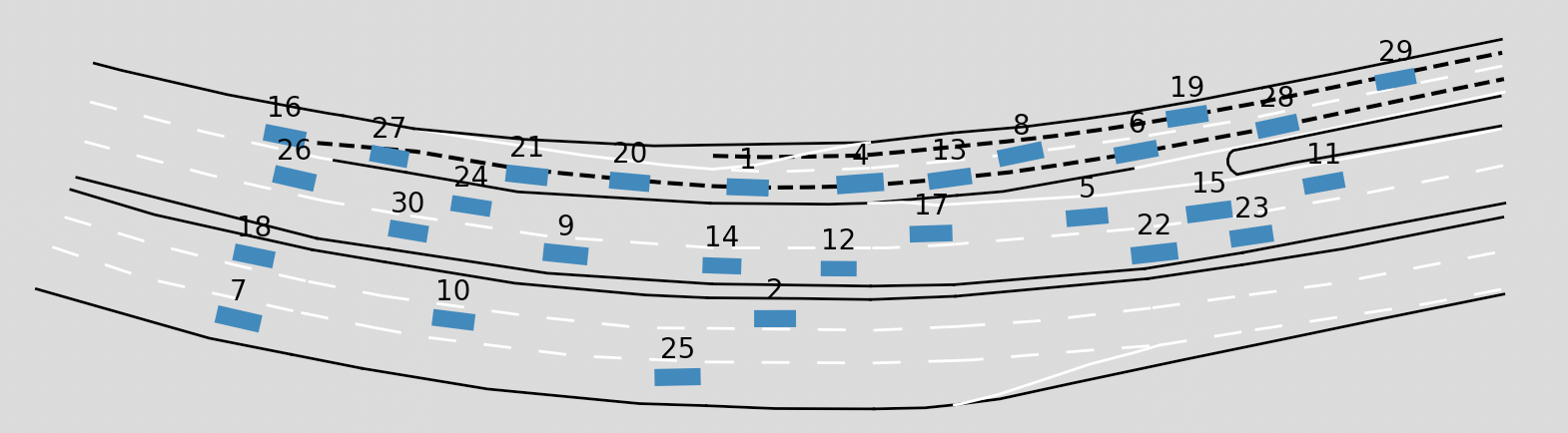}
    \caption{A snapshot of highway merging scenario in the INTERACTION data set. Vehicle $8$ is a merging vehicle between car $13$ and $6$.}
    \label{fig:interaction_example}
\end{figure*}

\subsubsection{Features} For this challenging and complex driving task, we propose features that capture efficiency (driving forward and achieving the lane merge for the merging vehicle), comfort and energy efficiency (penalizing large accelerations), and safety (penalizing relative speed and a Gaussian-cost on close proximity with other cars). In total, we learn the cost coefficients for $9$ features for both merging and follower vehicles.

To evaluate the learning results on the INTERACTION data set, since we do not know the true cost function, we cannot compute the corresponding cost distribution. Instead, we compute the root mean squared error on trajectories that are generated from the learned cost functions. For each trajectory tuple in the data set, we run $10$ trials using the ECE policy with the learned cost function from both MA-IRL and CIOC. Then, we compute the average deviation from true trajectories over 5 seconds. For position errors, MA-IRL achieves significantly better prediction accuracy than the CIOC model due to the interactive nature of this scenario. IDM's performance is roughly on par with MA-IRL. We note that IDM is a popular expert designed, and empirically validated driving model for traffic simulation, and is only applicable for car following scenarios. Hence, IDM acts as aproxy for the unknown ground truth policy in this scenario.  Moreover, we learn the parameters of IDM to specifically fit this data set, thereby directly approximating the control policy of the vehicles in the data set. 
\begin{figure}[t]
\centering
    \includegraphics[width=\columnwidth]{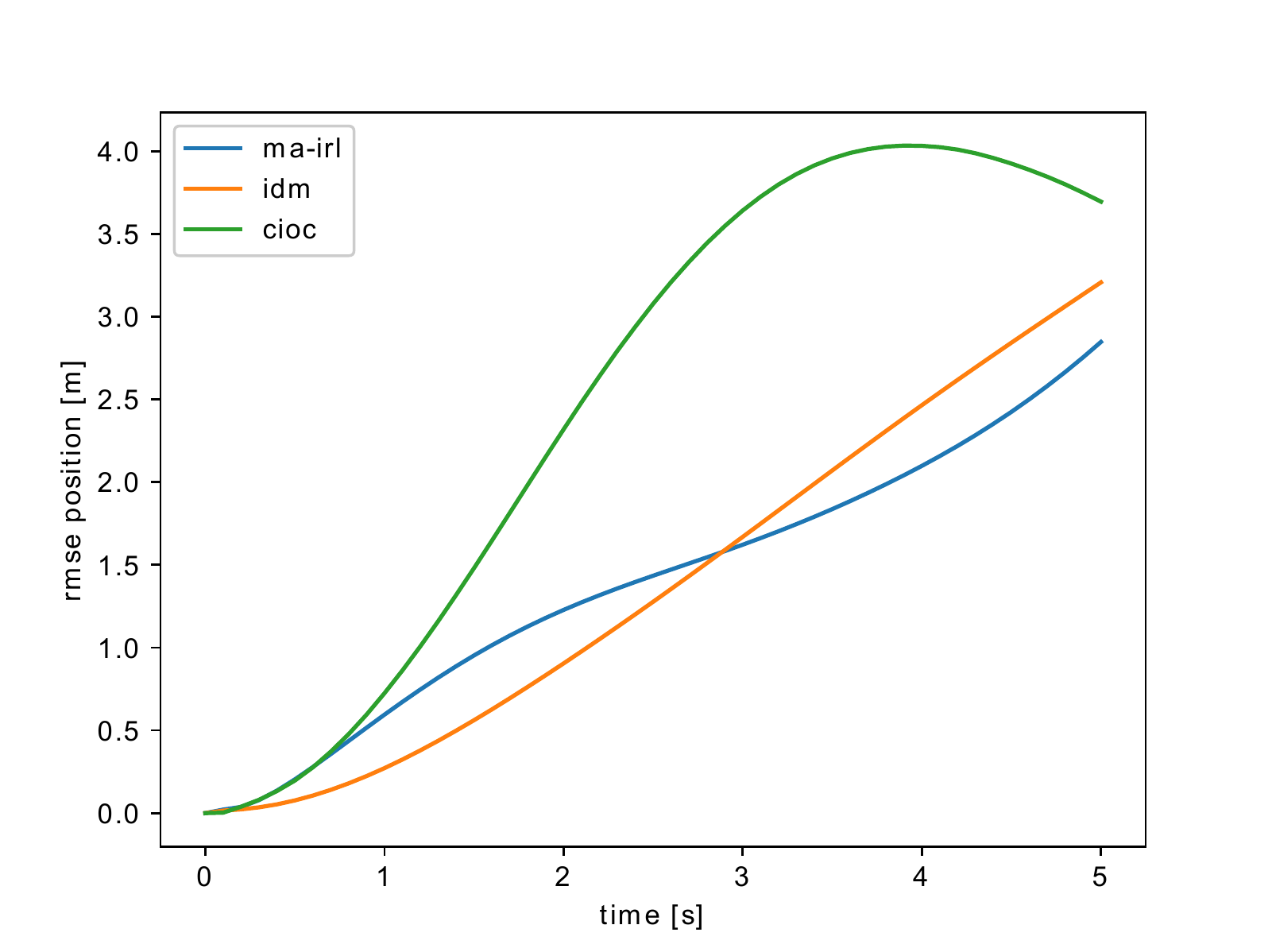}
    \caption{Root Mean Square Error }
    \label{fig:interaction_rmse}
\end{figure}

In addition, we follow previous work and evaluate the learned policy by computing task-related statistics from demonstrations and generated samples. We measure average speed of both merging and following cars, the average Euclidean distance between leader and merging cars, and average Euclidean distance between merging and following cars. The results are summarized in Table~\ref{tbl:interaction_statistics}. In Fig.~\ref{fig:interaction_rmse}, the position rmse of MA-IRL is similar to that of IDM and better than CIOC. From Table~\ref{tbl:interaction_statistics}, we can see that MA-IRL achieves good approximation results in average speed of merging vehicle and follower vehicle, which is on par with IDM's performance and better than CIOC. For average distance between merging vehicle and other vehicles, the performance of all three algorithms are similar and have larger deviation than speed statistics. We think the reason is that average distance is subject to more uncertainty.
\begin{table}[t]
\centering
\begin{tabular}{@{}l|c|c|c|c@{}}
\toprule
task statistics      & true  & MA-IRL  & IDM  & CIOC    \\ 
\midrule
M avg. speed [m/s]         & 3.272   & 3.372   & 2.876  & 2.949 \\ 
F avg. speed [m/s]        & 2.893   & 2.592   & 2.574  & 2.410 \\
ML avg. dist [m]        & 7.770   & 7.438   & 7.951  & 8.041 \\
MF avg. dist [m]        & 7.447   & 8.716   & 6.844  & 7.009 \\
\bottomrule
\end{tabular}
\caption{Task statistics for generated simulations using learned cost function/policy compared to human demonstrations from the same initial configuration.}
\label{tbl:interaction_statistics}
\end{table}

\section{Conclusion}\label{sec:concl}
In this paper, we presented an algorithm for Maximum Entropy Inverse Reinforcement Learning in multi-agent settings. To enable learning from boundedly rational agents such as humans, we defined a notion of noisy equilibrium called Entropic Cost Equilibrium (ECE). We proved that this noisy equilibrium is indeed an extension of the maximum entropy principle to the multi-agent setting. We then proved that ECE policies can be obtained in closed form for the special class of linear quadratic Gaussian games. We further presented an algorithm for approximating ECE policies for general nonlinear games. Finally, knowing how to find maximum-entropy multi-agent policies, we provided an iterative algorithm for learning agents' costs from a set of demonstrations when cost functions are represented as a linear combination of a set of features. We verified and validated our algorithm using both synthetic and real-world data set demonstrating that our MA-IRL algorithm can successfully capture the interactions between agents and learn accurate models of agents' costs.

\emph{Limitations and Future Work:} While our simulations and experiments show the effectiveness of our MA-IRL algorithm, a key future direction is to enable learning general nonlinear costs. In this paper, we assumed that agents' costs can be represented as weighted sum of a set of features. An important future direction is to relax this assumption. Moreover, to avoid handpicking the set of relevant features a-priori, we would like to study the extensions of this work to enable learning cost functions that are represented in the form of function approximators such as neural networks. Finally, the current work can be extended to the case of unknown system dynamics as well. We assumed that the system dynamics were known and fixed. We believe that our method can be further extended to systems with unknown dynamics.

\appendices
\section{Proof of Theorem~\ref{theorem:LQ-game}}\label{appendix-proof-thereom2}
We provide the proof for Theorem~\ref{theorem:LQ-game}. We know from Theorem~\ref{theorem:max-ent-equi} that to find the ECE policies of a linear quadratic game $G=(\cc{S},\bm{\cc{A}}, f, g, p_0, \bm{c})$, we can equivalently find the Nash equilibrium policies of the auxiliary game $\tilde{G}=(\cc{S},\bm{\cc{A}}, f, g, p_0, \tilde{\bm{c}},\horizon)$. The key idea is that at Nash equilibria of the game $\tilde{G}$, the cost-to-go or the value function for each agent at any given state can be represented in closed form via a quadratic function of the state. Using dynamic programming, the closed-form solution of the cost-to-go function can be propagated backwards in time to find both the Nash equilibirum policy and the cost-to-go of agent $i$. 

More precisely, let $\bm{\pi}^*$ denote a set of Nash equilibirum policies of the game $\tilde{G}$. Consider an agent $i \in \ageset$. Fix the policies of all the other agents $\bm{\pi}^{{-i}^*}$. Then, using Definition~\ref{def:nash}, for the game $\tilde{G}$, the (Nash) equilibirum policy of the agent $\pi^{i^*}$ optimizes 
\begin{align}\label{eq:nash-tilde}
\begin{split}
    \min_{\pi^i} \, \mathbb{E}_{\pi^i, \bm{\pi}^{{-i}^*}} \sum_{k=1}^{\horizon} \frac{1}{2}\Bigg( \state_k^\transpose Q^i \state_k &+ l_i^\transpose s_k +   \sum_{j\in \ageset}  {\action_k^j}^\transpose R^{ij} {\action^j_k}\Bigg)-\\ & \qquad \qquad \sum_{k=1}^\horizon \cc{H}(\pi^i_k(.|\state_t)).
\end{split}
\end{align}
First, note that for any policy $\pi^i(.|s_t)$, since the dynamics are linear and the cost is quadratic, the expectation of the terms inside the parenthesis in~\eqref{eq:nash-tilde}, depends only the first and second order moments of the policy $\pi^i$. Therefore, for any fixed first and second order moments of the policy $\pi^i$, the optimal policy $\pi^{i^*}$ will be the distribution with maximum entropy subject to fixed first and second order moments. Leveraging the fact that with fixed first and second order moments, the maximum entropy distribution is a normal distribution (see for instance, Theorem 8.6.5 in~\cite{cover2012elements}), we can conclude that the optimal policy for each agent $i$ at any time step $\pi^{i^*}_t$ is normally distributed. Hence, it remains to find the mean and covariance of the optimal policy at any given time step. For every agent $i$, let the mean and covariance of agent $i$'s policy be represented by $\mu^{i}_t$ and $\Sigma^{i}_t$ at every time step $t$.
The optimal solution to~\eqref{eq:nash-tilde} can be found using dynamic programming. Let $V_t^{i^*}(s_t)$ denote the cost-to-go of agent $i$ at Nash equilibrium at a given time step $t$
\begin{align}\label{eq:value-definition}
\begin{split}
    V_t^{i^*}(\state_t) := &\min_{\pi^i} \, \mathbb{E}_{\pi^i, \bm{\pi}^{{-i}^*}} \sum_{k=t}^{\horizon} \frac{1}{2}\Bigg( \state_k^\transpose Q^i \state_k + {l^i}^\transpose \state_k +\\ &\qquad \sum_{j\in \ageset}  {\action_k^j}^\transpose R^{ij} {\action^j_k}\Bigg) - \sum_{k=t}^\horizon \cc{H}(\pi^i_k(.|\state_k)).
    \end{split}
\end{align}

We prove via induction that the cost-to-go for each agent at a given time step is a quadratic function of the state. We assume for time $t+1$, we have
\begin{align}\label{eq:value-function-structure-1}
    V_{t+1}^{i^*}(\state_{t+1}) = \frac{1}{2}\state^\transpose Z_{t+1}^i \state_{t+1} + {\xi_{t+1}^i}^\transpose \state_{t+1} + n_{t+1}^i,
\end{align}
where $Z_{t+1}^i$, $\xi_{t+1}^i$, and $n_{t+1}^i$ are the coefficients of the quadratic cost-to-go at time $t+1$. 
We will propagate the cost-to-go structure~\eqref{eq:value-function-structure-1} backwards in time to prove that
\begin{align}\label{eq:value-function-structure}
    V_t^{i^*}(\state_t) = \frac{1}{2}\state^\transpose Z_{t}^i \state_t + {\xi_t^i}^\transpose \state_t + n_t^i.
\end{align}
We prove this via dynamic programming. For time step $t$, we have
\begin{align}\label{eq:value-function-t}
    V_t^{i^*}(\state_t) = \min_{\pi^i_t} \, \mathbb{E}_{\pi^i_t, \bm{\pi}_t^{{-i}^*}}  \frac{1}{2}&\Bigg( \state_t^\transpose Q^i \state_t +\sum_{j\in \ageset}  {\action_t^j}^\transpose R^{ij} {\action^j_t}\Bigg)- \nonumber \\ &  \cc{H}(\pi^i_t(.|\state_t))+\mathbb{E}_{s_{t+1}}  \, V_{t+1}^{i^*}(\state_{t+1}).
\end{align}
Using dynamics~\eqref{eq:dynamics-linear}, the structure of cost-to-go function for time $t+1$ from~\eqref{eq:value-function-structure-1}, and the fact that the policies for every time step are normally distributed, one can verify that

\begin{align}\label{eq:expected-value-function}
&\mathbb{E}\, V_{t+1}^{i^*}(\state_{t+1}) = \frac{1}{2} \Big[  \state_t^\transpose A^\transpose Z_{t+1}^i A \state_t + \sum_{j \in \ageset} {\mu^j_t}^\transpose {B^j}^\transpose Z_{t+1}^i A \state_t +\nonumber \\ &\state_t^\transpose A^\transpose Z_{t+1}^i \sum_{j \in \ageset} B^j \mu_t^j + \sum_{j \in \ageset} \left(B^j \mu_t^j \right)^\transpose Z_{t+1}^i  + n_{t+1}^i + \text{tr}\Big( \nonumber \\ & Z_{t+1}^i \sum_{j\in \ageset} B^j \Sigma_t^j {B^j}^\transpose + Z_{t+1}^i \Big) + {\xi^i}^\transpose_{t+1} \big(A\state_t + \sum_{j\in\ageset} B^j \mu_t^j \big) \Big].   
\end{align}
Moreover, we know that for a normal policy $\pi_t^i$ with mean $\mu^i_t$ and covariance $\Sigma^i_t$, the entropy of the policy is 
\begin{align}\label{eq:entropy-policy}
    \cc{H}(\pi^i_t (.|\state_t)) = \frac{m^i}{2} \log (2 \pi e) + \frac{1}{2} \log \det (\Sigma_t^i),
\end{align}
where $m^i$ is the dimension of the action space of agent $i$, and $e$ is the Euler's number. Using~\eqref{eq:expected-value-function} and~\eqref{eq:entropy-policy},~\eqref{eq:value-function-t} which is the cost-to-go at time $t$,~\eqref{eq:value-function-t} can be rewritten as a minimization with respect to $\mu^i_t$ and $\Sigma^i_t$ when we fix the mean and covariance of other agents' policies
\begin{align}\label{eq:value-function-simple}
    &V_t^{i^*}(s_t)= \min_{\mu^i, \Sigma^i} \, \mathbb{E}_{\pi^i, \bm{\pi}^{{-i}^*}} \frac{1}{2} [\state_t^\transpose Q^i \state_t + \sum_{j\in \ageset}\mu_t^{j^\transpose} R^{ij} \mu_t^{j} + \nonumber \\&\state_t^\transpose A^\transpose Z_{t+1}^i A \state_t + \sum_{j \in \ageset} \mu_t^j B^{j^\transpose} Z_{t+1}^i A \state_t + \nonumber \\
    &\state_t^\transpose A^\transpose Z^i_{t+1} \sum_j B^j \mu^j + (\sum_{j\in\ageset} B^j \mu^j_t )^\transpose Z_{t+1}^i (\sum_{j\in\ageset} B^j \mu^j_t )]+ \nonumber \\ &\xi^{i^\transpose}_{t+1} (A\state_t+\sum_{j\in \ageset} B^j \mu^j )+ n_{t+1}^i - \frac{1}{2}\log\det(\Sigma_t^i) - \frac{1}{2} m^i 2\pi e \nonumber \\ & +\frac{1}{2}\sum_{j\in \ageset} \text{tr}(R^{ij}\Sigma_t^j + Z_{t+1}^i(\sum_{j\in \ageset} B^j\Sigma_t^j B^{j^\transpose})+Z_{t+1}^i).
\end{align}
The objective function of~\eqref{eq:value-function-simple} is a convex quadratic function of $\mu^i_t$ and a convex function of $\Sigma^i_t$. Taking the derivative of the objective function in~\eqref{eq:value-function-simple} with respect to $\mu^i_t$ and $\Sigma^i_t$ and setting the derivative equal to zero, one can verify that $\mu^{i^*}_t$ and $\Sigma^{i^*}_t$ have the structure of~\eqref{eq:mu} and~\eqref{eq:sigma}. By replacing $\mu^{i^*}_t$ and $\Sigma_t^{i^*}$ in~\eqref{eq:value-function-simple} and rearranging the terms, one can see that $V_t^{i^*}(s_t)$ is also a quadratic function of the state $\state_t$. This proves that at every time step $t$, $V_t^{i^*}(\state_t)= \frac{1}{2}\state^\transpose Z_{t}^i \state_t + {\xi_t^i}^\transpose \state_t + n_t^i$. By matching the coefficients of this quadratic function, recursions~\eqref{eq:Z-update},~\eqref{eq:xi-update},~\eqref{eq:F-update} and~\eqref{eq:beta_update} are obtained.
Note that we also need to verify the base case for our induction, i.e, the cost-to-go at the final time step $V^{i^*}_T(s_T)$ is quadratic. At the final time step $T$, the optimal policy $\pi^{i^*}_T$ is a zero mean normal distribution with covariance matrix $\Sigma_T^i = {(R^i)}^{-1}$. Plugging in this policy, we verify that in the final time step the cost-to-go is also quadratic in state. This proves the base case for our induction, which completes our proof that the optimal cost-to-go at any time step is quadratic in states. 
Now that we have proved that at every time step, $V_t^{i^*}(s_t)$ is of the form~\eqref{eq:value-function-structure}, we can obtain  the recursions~\eqref{eq:P-recur} and~\eqref{eq:alpha-recur} on the matrix $P^i_t$ and the vector $\alpha^i_t$ by replacing~\eqref{eq:mu} and~\eqref{eq:sigma} in~\eqref{eq:value-function-t} and matching the coefficients.


\ifCLASSOPTIONcaptionsoff
  \newpage
\fi







\end{document}